%% file: discontmm_L2analysis.tex
\documentclass[12pt]{article}

\usepackage[longpaper]{preprint-layout}
\usepackage{preprint-notation}
\usepackage{preprint-tools}
\usepackage{hyperref}

\usepackage{mathrsfs}
\usepackage{dsfont}
\usepackage{import}
\usepackage{wrapfig}

\title{Space-Time CutFEM on Overlapping Meshes: Simple Discontinuous Mesh Evolution}
\author{Mats G. Larson\footnote{\addressumu} \mbox{} Carl Lundholm\footnote{\addressumu}}
\date{\today}

\numberwithin{equation}{section}

\newtheorem{theorem}{Theorem}[section] 
\newtheorem{lemma}{Lemma}[section]
\newtheorem{corollary}{Corollary}[section]

\newcommand{\bs}{\bar{s}}

\newcommand{\ud}{\,\mathrm{d}}
\newcommand{\nab}{\nabla}

\newcommand{\lap}{\Delta}

\newcommand{\euler}{\text{e}}
\newcommand{\norma}[1]{\left|\mkern-1.5mu\left|\mkern-1.5mu\left| #1 \right|\mkern-1.5mu\right|\mkern-1.5mu\right|_{A_{h,t}}}
\newcommand{\normam}[1]{\left|\mkern-1.5mu\left|\mkern-1.5mu\left| #1 \right|\mkern-1.5mu\right|\mkern-1.5mu\right|_{A_{m}}}
\newcommand{\norman}[1]{\left|\mkern-1.5mu\left|\mkern-1.5mu\left| #1 \right|\mkern-1.5mu\right|\mkern-1.5mu\right|_{A_{n}}}
\newcommand{\normanm}[1]{\left|\mkern-1.5mu\left|\mkern-1.5mu\left| #1 \right|\mkern-1.5mu\right|\mkern-1.5mu\right|_{A_{n-1}}}

\newcommand{\normaspecmn}[1]{\left|\mkern-1.5mu\left|\mkern-1.5mu\left| #1 \right|\mkern-1.5mu\right|\mkern-1.5mu\right|_{\mathscr{A}_{m,n}}}
\newcommand{\normaspecnm}[1]{\left|\mkern-1.5mu\left|\mkern-1.5mu\left| #1 \right|\mkern-1.5mu\right|\mkern-1.5mu\right|_{\mathscr{A}_{n,m}}}

\newcommand{\Gt}{\Gamma(t)}

\newcommand{\bGn}{\bar{\Gamma}_n}

\newcommand{\Om}[1]{{\Omega_#1}}
\newcommand{\om}[1]{{\omega_#1}}
\newcommand{\Omt}[1]{{\Omega_#1(t)}}

\usetikzlibrary{arrows}
\usetikzlibrary{shapes}

\begin{document}

\maketitle

\begin{abstract}
We present a cut finite element method for the heat equation on two overlapping meshes: a stationary background mesh and an overlapping mesh that evolves inside/``on top'' of it. Here the overlapping mesh is prescribed a simple discontinuous evolution, meaning that its location, size, and shape as functions of time are \emph{discontinuous} and \emph{piecewise constant}. For the discrete function space, we use continuous Galerkin in space and discontinuous Galerkin in time, with the addition of a discontinuity on the boundary between the two meshes. The finite element formulation is based on Nitsche's method. The simple discontinuous mesh evolution results in a space-time discretization with a slabwise product structure between space and time which allows for existing analysis methodologies to be applied with only minor modifications. We follow the analysis methodology presented by Eriksson and Johnson in~\cite{Eriksson1991, Eriksson1995}. The greatest modification is the introduction of a Ritzlike ``shift operator'' that is used to obtain the discrete strong stability needed for the error analysis. The shift operator generalizes the original analysis to some methods for which the discrete subspace at one time does not lie in the space of the stiffness form at the subsequent time. The error analysis consists of an a priori error estimate that is of optimal order with respect to both time step and mesh size. We also present numerical results for a problem in one spatial dimension that verify the analytic error convergence orders.
\end{abstract}

\vspace{1cm}

%\cleardoublepage

\section{Introduction}

\paragraph{Issue - Cost of mesh generation:} Generating computational meshes for numerically solving differential equations can be a computationally costly procedure. In practical applications the mesh generation can often represent a substantial amount of the total computation time. This is especially true for problems where the solution domain changes during the solve process, e.g., evolving geometry and shape optimization. With standard methods the mesh then has to be constantly checked for degeneracy and updated if needed, meaning a persisting meshing cost for the entire solve process. 

\paragraph{Remedy - CutFEM:} Cut finite element methods (CutFEMs) provide a way of decoupling the computational mesh from the problem geometry. This means that the same discretization can be used for a changing solution domain. CutFEMs can thus make remeshing redundant for problems with changing geometry but also for other applications involving meshing such as adaptive mesh refinement. The cost of CutFEMs is treating the mesh cells that are arbitrarily cut by the independent problem geometry.

\begin{wrapfigure}{r}{0.5\textwidth}
\centering
\includegraphics[width=0.5\textwidth]{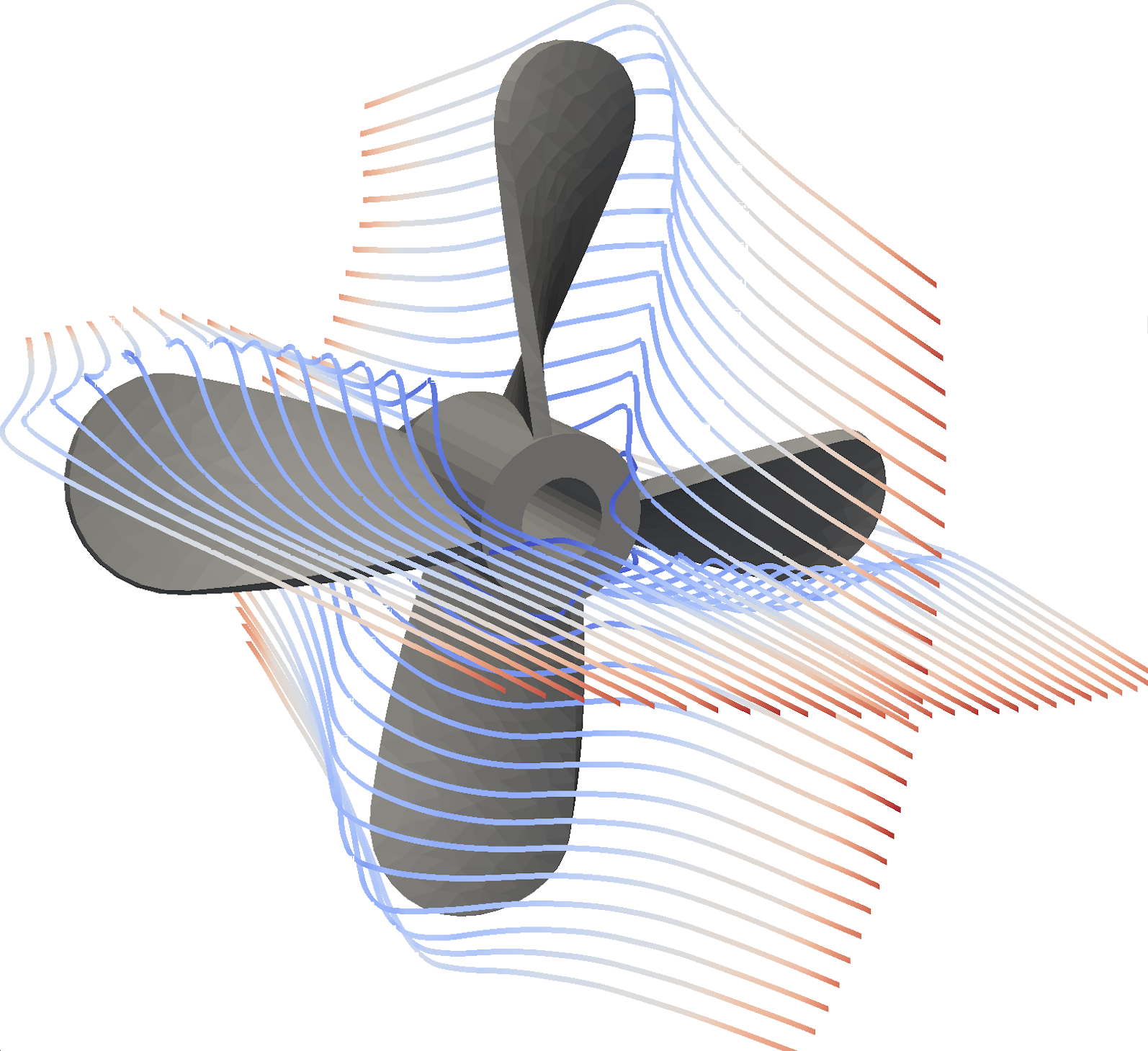}
\caption{Computed streamlines around a propeller. \href{https://link.springer.com/article/10.1186/s40323-015-0043-7}{Image} by \href{https://anders.logg.org/}{Anders Logg} is licensed under \href{https://creativecommons.org/licenses/by/4.0/}{CC BY 4.0}.}
\label{fig_propeller_flow}
%\vspace{\baselineskip}
\end{wrapfigure}

\paragraph{CutFEM on overlapping meshes:} A common type of problem with changing geometry is one where there is a moving object in the solution domain, e.g., see Figure~\ref{fig_propeller_flow}. A straightforward CutFEM-approach to this problem would be to consider CutFEM for the interface problem, i.e., to use a background mesh of the empty solution domain together with an interface that represents the object. However, a more advantageous and sophisticated approach is to consider CutFEM on \emph{overlapping meshes}, meaning two or more meshes ordered in a mesh hierarchy. This is also called composite grids/meshes and multimesh in the literature. The idea is to use a background mesh of the empty solution domain, just as for the interface problem, but instead to encapsulate the object in a second mesh. The mesh containing the object is then placed ``on top'' of the background mesh, creating a mesh hierarchy. The motion of the object will thus also cause its encapsulating mesh to move. There are some advantages of using a second overlapping mesh instead of an interface. Firstly, an overlapping mesh can incorporate boundary layers close to the object. Something an interface cannot. Secondly, the total number of degrees of freedom (DOFs) of the resulting linear system may be reduced. This is so since for CutFEM for the interface problem this number can be twice the number of DOFs of the background mesh or more, whereas for CutFEM on overlapping meshes it will be the number of DOFs of the background mesh plus the number of DOFs of the second mesh. Thirdly, if the object has a complicated geometry, representing it with an interface can lead to tricky cut situations and thus a higher computational cost. By instead using an object-encapsulating mesh with a simply-shaped exterior boundary, the cut situations can be made less tricky, see Figure~\ref{fig_propeller_om}. A way to further sophisticate this is to allow the moving object to deform the interior of the overlapping mesh while initially keeping its exterior boundary fixed. Only when the deformations have become too large is the overlapping mesh ``snapped'' into place to avoid degeneracy. Such a snapping feature provides a choice between computing cut situations or computing deformations, thus allowing the cheapest option for the situation at hand to be chosen. A drawback of using a second overlapping mesh instead of an interface is that overlapping meshes require collision computations between the cells of the meshes, something that can be computationally expensive.

\begin{wrapfigure}{r}{0.5\textwidth}
\centering
\includegraphics[width=0.5\textwidth]{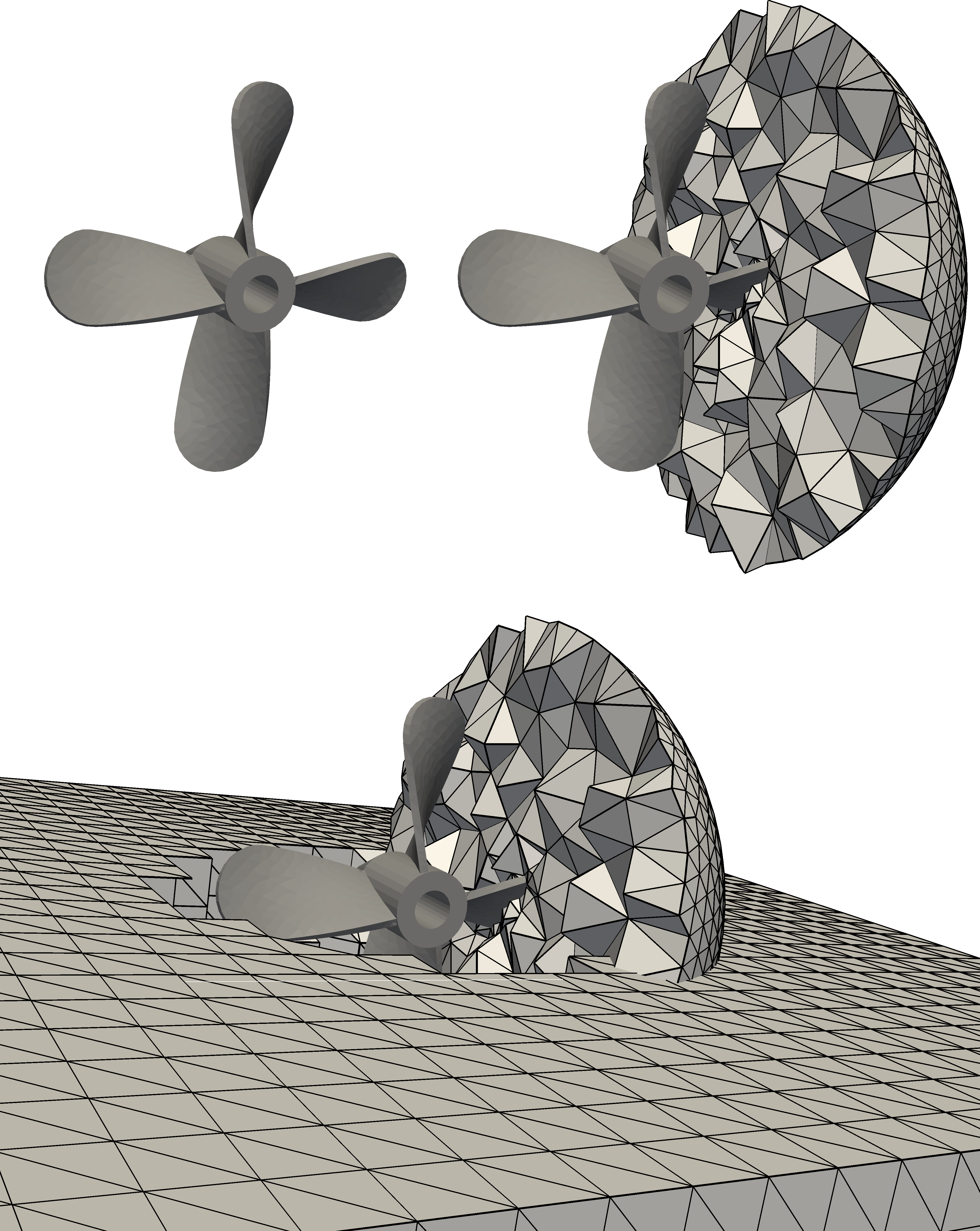}
\caption{Overlapping meshes for a problem with a rotating propeller. \href{https://link.springer.com/article/10.1186/s40323-015-0043-7}{Image} by \href{https://anders.logg.org/}{Anders Logg} is licensed under \href{https://creativecommons.org/licenses/by/4.0/}{CC BY 4.0}.}
\label{fig_propeller_om}
%\vspace{-\baselineskip}
\end{wrapfigure}

CutFEM on overlapping meshes can also be used as an alternative to adaptive mesh refinement by keeping a smaller finer mesh in regions requiring higher accuracy. Yet another application example is to use a composition of simpler structured meshes to represent a complicated domain.

\paragraph{Literary background:} Over the past two decades, a theoretical foundation for the formulation of
stabilized CutFEM has been developed by extending the ideas
of Nitsche, presented in~\cite{Nitsche:1971aa}, to a general weak formulation of the
interface conditions, thereby removing the need for domain-fitted
meshes. The foundations of CutFEM were presented
in~\cite{Hansbo:2002aa} and then extended to overlapping meshes in~\cite{Hansbo:2003aa}. The CutFEM methodology has since been developed and applied to a number of important multiphysics problems.
See for example~\cite{Burman:2007ab,Burman:2007aa,Becker:2009aa,Massing2014a}. For overlapping meshes in particular, see for example~\cite{Massing2014,Johansson:2015aa, Dokken2019, Johansson2019}. So far, only CutFEM for \emph{stationary} problems on overlapping meshes have been developed and analysed to a satisfactory degree, thus leaving analogous work for \emph{time-dependent} problems to be desired.

\paragraph{This work:} The work presented here is intended to be an initial part of developing and analysing CutFEMs for time-dependent problems on overlapping meshes. We consider a CutFEM for the heat equation on two overlapping meshes: one stationary background mesh and one moving overlapping mesh with no object. Depending on how the mesh motion is represented discretely, quite different space-time discretizations may arise, allowing for different types of analyses to be applied. Generally the mesh motion may either be continuous or discontinuous. We have considered the simplest case of both of these two types, which we refer to as \emph{simple continuous} and \emph{simple discontinuous} mesh motion. Simple continuous mesh motion means that the location of the overlapping mesh as a function of time is \emph{continuous} and \emph{piecewise linear}, and simple discontinuous mesh motion means that it is \emph{discontinuous} and \emph{piecewise constant}. The former is studied in other work and the latter in this. Here, with no extra effort, we may also extend the mesh motion to a mesh evolution, meaning that size and shape change as well. Thus, simple discontinuous mesh evolution means that the location, size, and shape of the overlapping mesh as functions of time are \emph{discontinuous} and \emph{piecewise constant}. 

%that is a Ritz projection operator. This is somewhat similar to what is done in~\cite{MaZhangZheng2022}

\paragraph{Analytic novelty:} The simple discontinuous mesh evolution results in a space-time discretization with a slabwise product structure between space and time. Standard analysis methodology therefore work with some modifications. We follow the analysis presented by Eriksson and Johnson in~\cite{Eriksson1991, Eriksson1995}. The main modification is the introduction of what we call a ``shift operator'' which can be viewed as a modified Ritz projection operator. This is somewhat similar to what is done in~\cite{MaZhangZheng2022}. The shift operator is used to obtain the discrete strong stability needed for the error analysis. In the proof of the strong stability, discrete functions on one slab need to be translated into discrete functions on the subsequent slab. In~\cite{Eriksson1991}, this is done by simply assuming that the discrete space of one slab lies in the discrete space of the next. In~\cite{Eriksson1995}, the proof is generalized by mapping discrete functions of one slab to the next with the Ritz projection operator. Here, the Ritz projection on one slab is not defined for discrete functions on another slab since the discontinuity on the interface between the two meshes changes between slabs. This issue is solved by instead using the shift operator to map between discrete subspaces. The shift operator thus further generalizes the original analysis to be applicable to some methods where the discrete space of one slab does not lie in the space of the stiffness form on the subsequent slab, e.g., problems with changing interior geometry. The error analysis concerns an optimal order a priori error estimate of the $L^2$-norm of the approximation error at the final time. This estimate shows that the method preserves the superconvergence of the error with respect to the time step.

\paragraph{Paper overview:} In Section 2, the model problem is formulated. In Section 3, the CutFEM is presented. In Section 4, tools for the analysis are presented including the shift operator. In Section 5, we present and prove stability estimates for the discrete solution. In Section 6, we present and prove an optimal order a priori error estimate. In Section 7, we present numerical results for a problem in one spatial dimension that verify the analytic convergence orders. In the appendix we present technical estimates and interpolation results used in the analysis.

\section{Problem} \label{sec:probform}

\begin{wrapfigure}{r}{0.5\textwidth}
\centering
\def\svgwidth{0.5\textwidth}
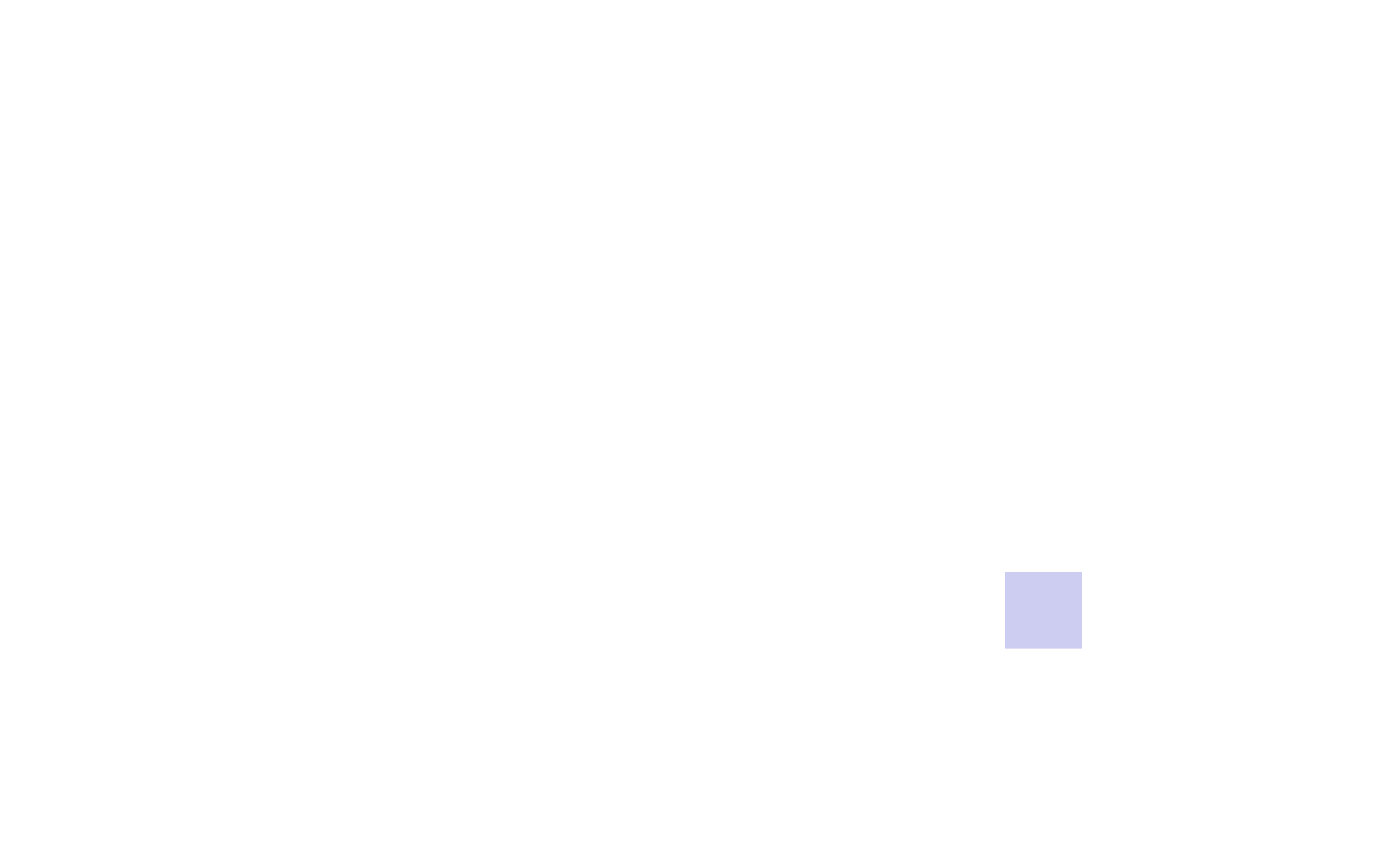
\caption{Partition of $\Om{0}$ into $\Om{1}$ (blue) and $\Om{2}$ (red) for $d = 2$ with $G$ moving with velocity $\mu$.}
\label{fig_probdom}
%\vspace{-\baselineskip}
\end{wrapfigure}
For $d = 1, 2$, or $3$, let $\Om{0} \subset \mathbb{R}^d$ be a bounded convex domain with polygonal boundary $\partial\Om{0}$. Let $T > 0$ be a given final time. Let $G \subset \Om{0} \subset \mathbb{R}^d$ be another bounded domain with polygonal boundary $\partial G$. We let the location, size, and shape of $G$ be time-dependent by prescribing for $G$ a time-dependent spatially smooth velocity field $\mu_t : G(t) \to \mathbb{R}^d$. Using $\Om{0}$ and $G$, we define the following two domains:
\begin{align}
\Om{1} &:= \Om{0} \setminus (G \cup \partial G) \label{def:Om1} \\
\Om{2} &:= \Om{0} \cap G \label{def:Om2}
\end{align}
with boundaries $\partial\Om{1}$ and $\partial\Om{2}$, respectively. Let their common boundary be
\begin{equation}
\Gamma := \partial\Om{1} \cap \partial\Om{2} \label{def:Gamma}
\end{equation}
For $t \in [0, T]$, we have the partition
\begin{equation}
\Om{0} = \Omt{1} \cup \Gamma(t) \cup \Omt{2} \label{partitionOm0}
\end{equation}
See Figure~\ref{fig_probdom} for an illustration. We consider the heat equation in $\Om{0} \times (0, T]$ with source $f \in L^2((0, T], \Om{0})$, homogeneous Dirichlet boundary conditions, and initial data $u_0 \in H^2(\Om{0}) \cap H_0^1(\Om{0})$:
\begin{equation} 
\left\{
\begin{split}
\dot{u} - \lap{u}  & = f && \text{in} \ \Om{0} \times (0, T] \\
u & = 0 && \text{on} \ \partial\Om{0} \times (0, T] \\
u & = u_0 && \text{in} \ \Om{0} \times \{0\} 
\end{split}
\right. \label{heateq}
\end{equation} 

%\cleardoublepage

\section{Method} \label{sec:feform}

\subsection{Preliminaries}

Let $\mathcal{T}_0$ and $\mathcal{T}_G$ be quasi-uniform simplicial meshes of $\Om{0}$ and $G$, respectively. We denote by $h_K$ the diameter of a simplex $K$. We partition the time interval $(0, T]$ quasi-uniformly into $N$ subintervals $I_n = (t_{n-1}, t_n]$ of length $k_n = t_n - t_{n-1}$, where $0 = t_0 < t_1 < \ \dots \ < t_N = T$ and $n = 1, \dots , N$. We assume the following space-time quasi-uniformity: For $h = \max_{K \in \mathcal{T}_0 \cup \mathcal{T}_G}\{h_K\}$, and $k = \max_{1 \leq n \leq N}\{k_n\}$,
\begin{equation}
h \lesssim k_{\min} \quad \quad k \lesssim h_{\min}
\label{quasiuniformity_st}
\end{equation}
where $k_{\min} = \min_{1 \leq n \leq N}\{k_n\}$, and $h_{\min} = \min_{K \in \mathcal{T}_0 \cup \mathcal{T}_G}\{h_K\}$. We note that this space-time quasi-uniformity is stricter than the one assumed in~\cite{Eriksson1995}, which here has the equivalent form $h^2 \lesssim k_{\min}$. The stricter one is needed because of using the shift operator results Lemma~\ref{lem:shiftop_error} and Lemma~\ref{lem:shiftenergy} in the analysis. We next define the following slabwise space-time domains:
\begin{align}
S_{{0,n}}  &:= \Om{0} \times I_n \label{def:S0n} \\
S_{{i,n}} &:= \{(x, t) \in S_{{0,n}}  : x \in \Omt{i}\} \label{def:Sin} \\
\bGn &:= \{(s, t) \in S_{{0,n}}  : s \in \Gamma(t)\} \label{def:bGn}
\end{align}
\begin{wrapfigure}{r}{0.5\textwidth}
\centering
\def\svgwidth{0.5\textwidth}
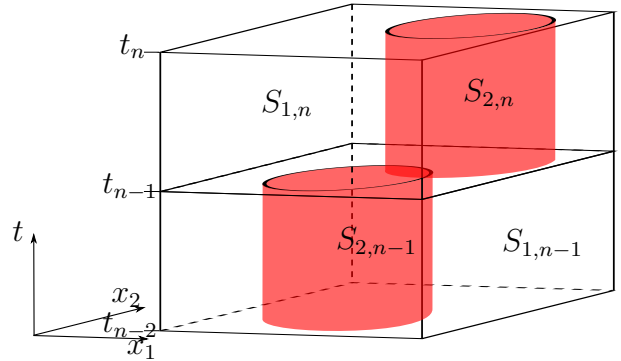
\caption{Space-time slabs with simple discontinuous mesh motion.}
\label{fig_spacetimeslabs_dG0mesh}
\vspace{-\baselineskip}
\end{wrapfigure}
In general we will use bar, i.e., $\bar{\cdot}$, to denote something related to space-time, such as domains and variables.
In addition to the domains $\Omt{1}$ and $\Omt{2}$, we also consider the ``covered'' overlap domain $\Omt{O}$. To define it, we will use the set of simplices $\mathcal{T}_{0, \Gt} := \{K \in \mathcal{T}_0 : K \cap \Gt \neq \emptyset \}$, i.e., all simplices in $\mathcal{T}_0$ that are cut by $\Gamma$ at time $t$. We define the overlap domain $\Omt{O}$ for a time $t \in [0, T]$ by
\begin{equation}
\Omt{O} := \bigcup_{K \in \mathcal{T}_{0, \Gt}} K \cap \Omt{2}
\label{def:OmOt}
\end{equation}
As a discrete counterpart to the evolution of the domain $G$, we prescribe a simple discontinuous evolution for the overlapping mesh $\mathcal{T}_G$. By this we mean that the location, size, and shape of the overlapping mesh $\mathcal{T}_G$ are functions with respect to time that are \emph{discontinuous} on $[0, T]$ and \emph{constant} on each $I_n$. This means that on each $I_n$ the location, size, and shape of $\mathcal{T}_G$ are fixed, but change from $I_{n-1}$ to $I_n$. We simply take $\mathcal{T}_G$ on $I_n$ to be a mesh of $G(t_n)$. An illustration of the slabwise space-time domains $S_{i, n}$ defined by \eqref{def:Sin} is shown in Figure~\ref{fig_spacetimeslabs_dG0mesh}. The simple discontinuous mesh evolution results in the following: For $n = 1, \dots, N$
%
%\noindent As a discrete counterpart to the motion of the domain $G$, we prescribe a simple discontinuous motion for the overlapping mesh $\mathcal{T}_G$. By this we mean that the location of the overlapping mesh $\mathcal{T}_G$ is a function with respect to time that is \emph{discontinuous} on $[0, T]$ and \emph{constant} on each $I_n$. This means that on each $I_n$ the position of $\mathcal{T}_G$ is fixed, but changes from $I_{n-1}$ to $I_n$. We take this change to be $\int_{I_n} \mu(t) \ud t$, i.e., the total change in the location of $G$ over $I_n$. An illustration of the slabwise space-time domains $S_{i, n}$ defined by \eqref{def:Sin} is shown in Figure~\ref{fig_spacetimeslabs_dG0mesh}. The simple discontinuous mesh motion results in the following: For $n = 1, \dots, N$
%
\begin{subequations}
\begin{align}
\Om{{i,n}} &= \Om{i}(t_n) = \Omt{i} \quad \forall t \in I_n \label{def:Omin} \\
\Gamma_n &= \Gamma(t_n) = \Gamma(t) \quad \forall t \in I_n \label{def:Gamman} \\
\Om{{O,n}} &= \Om{O}(t_n) = \Omt{O} \quad \forall t \in I_n \label{def:OmOn}
\end{align}
\label{slabwise_s_domains}
\end{subequations}
%
%\begin{figure}[h]
%\centering
%\def\svgwidth{0.52\textwidth}
%\import{figures/}{spacetimeslabs_dG0mesh.pdf_tex}
%\caption{Space-time slabs with simple discontinuous mesh motion.}
%\label{fig_spacetimeslabs_dG0mesh}
%\end{figure}
%

\subsection{Finite element spaces} \label{subsec:fespaces}

We define the discrete spatial finite element spaces $V_{h,0}$ and $V_{h,G}$ as the spaces of continuous piecewise polynomials of degree $\le p$ on $\mathcal{T}_0$ and $\mathcal{T}_G$, respectively. We also let the functions in $V_{h,0}$ be zero on $\partial\Om{0}$. For $t \in [0, T]$, we use these two spaces to define the broken finite element space $V_h(t)$ by    
\begin{equation}
\begin{split}
V_h(t) := \{v : v|_{\Omt{1}} &= v_0|_{\Omt{1}} \text{ for some } v_0 \in V_{h,0} \text{ and } \\
v|_{\Omt{2}} &= v_G|_{\Omt{2}} \text{ for some } v_G \in V_{h,G} \}
\end{split} \label{fesvht}
\end{equation}
See Figure~\ref{figfefuncspace} for an illustration of a function $v \in V_h(t)$. From the simple discontinuous evolution of $\mathcal{T}_G$, we have via (\ref{def:Omin}) that: For $n = 1, \dots, N$
\begin{equation}
V_{h,n} = V_h(t_n) = V_h(t) \quad \forall t \in I_n
\label{slabwise_s_space}
\end{equation}
\begin{figure}[h]
\centering
\def\svgwidth{0.5\textwidth}
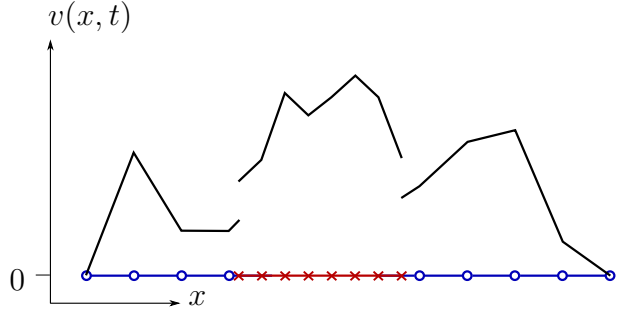
\caption{Example of $v(\cdot, t) \in V_h(t)$ for $d = 1$ and $p = 1$, where $\mathcal{T}_0$ is blue and $\mathcal{T}_G$ red.}
\label{figfefuncspace}
\end{figure}
\noindent For $n = 1, \dots, N$, we define the discrete space-time finite element spaces $V_{h,0}^n$ and $V_{h,G}^n$ as the spaces of functions that for a $t \in I_n$ lie in $V_{h,0}$ and $V_{h,G}$, respectively, and in time are polynomials of degree $\le q$ along the trajectories of $\mathcal{T}_0$ and $\mathcal{T}_G$ for $t \in I_n$, respectively. For $n = 1, \dots, N$, we use these two spaces to define the broken finite element space $V_h^n$ by:
\begin{equation}
\begin{split}
V_h^n := \{v : v|_{S_{1,n}} & = v_0^n|_{S_{1,n}} \text{ for some } v_0^n \in V_{h,0}^n \text{ and } \\ 
v|_{S_{2,n}} & = v_G^n|_{S_{2,n}} \text{ for some } v_G^n \in V_{h,G}^n \}
\end{split} \label{fesvhn}
\end{equation} 
We define the global space-time finite element space $V_h$ by: 
\begin{equation}
V_h := \{v : v|_{S_{0,n}} \in V_h^n, n = 1, \dots, N \} \label{fesvh}
\end{equation}

\subsection{Finite element formulation} \label{subsec:feform}

We may now formulate the space-time cut finite element formulation for the problem described in Section \ref{sec:probform} as follows: Find $u_h \in V_h$ such that
\begin{equation}
B_h(u_h, v) = \int_{0}^T (f, v)_{\Om{0}} \ud t + (u_{0}, v_{0}^+)_{\Om{0}} \quad \forall v \in V_h
\label{feform}
\end{equation}
The non-symmetric bilinear form $B_h$ is defined by
\begin{equation}
B_h(w, v) := \sum_{n=1}^N \int_{I_n} (\dot w, v)_{\Om{0}} + A_{h,t}(w, v) \ud t + \sum_{n=1}^{N-1}([w]_{n}, v_{n}^+)_{\Om{0}} + (w_{0}^+, v_{0}^+)_{\Om{0}}
\label{Bhdef}
\end{equation}
where $( \cdot , \cdot )_{\Omega}$ is the $L^2(\Omega)$-inner product, $[v]_n$ is the jump in $v$ at time $t_n$, i.e., $[v]_n = v_n^+ - v_n^-$, $v_n^\pm = \lim_{\varepsilon \to 0+} v(x, t_n \pm \varepsilon)$. The symmetric bilinear form $A_{h,t}$ is defined by
\begin{equation}
\begin{split}
A_{h,t}(w, v) := & \sum_{i=1}^2 (\nab w, \nab v)_{\Omt{i}} - (\langle \partial_{n} w \rangle, [v])_{\Gamma(t)} - (\langle \partial_{n} v \rangle, [w])_{\Gamma(t)} \\
& + (\gamma h_K^{-1}[w],[v])_{\Gamma(t)} + ([\nab w],[\nab v])_{\Omt{O}} 
\end{split}
\label{Ahdef}
\end{equation}  
where $\langle v \rangle$ is a convex-weighted average of $v$ on $\Gamma$, i.e., $\langle v \rangle = \omega_1v_1 + \omega_2v_2$, where $\omega_1, \omega_2 \in [0, 1]$ and $\omega_1 + \omega_2 = 1$, $v_i = \lim_{\varepsilon \to 0+} v(\bs - \varepsilon n_i)$, $\bs = (s,t)$, $n$ is the normal vector to $\Gamma$ (not to be confused with time index $n$, e.g., in $t_n$), $\partial_n v = n \cdot \nabla v$, $[v]$ is the jump in $v$ over $\Gamma$, i.e., $[v]= v_1 - v_2$, $\gamma \geq 0$ is a stabilization parameter, $h_K = h_K(x) = h_{K_0}$ for $x \in K_0$, where $h_{K_0}$ is the diameter of simplex $K_0 \in \mathcal{T}_0$, and $\Omt{O}$ is the overlap region defined by \eqref{def:OmOt}.

\section{Analytic preliminaries}

\subsection{The bilinear form $A_{h,t}$}
The space of $A_{h, t}$ is $H^{3/2 + \varepsilon}(\cup_i \Omt{i})$ where $\varepsilon > 0$ may be arbitrarily small.
%Let the space of $A_{h, t}$ be defined by
%%
% \begin{equation}
% H^1_{\partial_n \Gt}(\cup_i \Omt{i}) := \{v \in H^1(\cup_i \Omt{i}) : \partial_n v_i |_{\Gt} \in L^2(\Gt), \text{ for } i = 1, 2 \} 
% \end{equation}
% 
%For $t \in [0, T]$ and $k \in \mathbb{N}$ we define the broken Sobolev spaces 
%%
%\begin{equation}
%\begin{split}
%H^k(\Omt{1}, \Omt{2}) &:= H^k(\cup_{i}\Omt{i}) := \{ v \in L^2(\Om{0}) : v |_{\Omt{i}} \in H^k(\Omt{i}), i = 1, 2 \}, \\
%H_0^k(\Omt{1}, \Omt{2}) &:= H_0^k(\cup_{i}\Omt{i}) := \{ v \in H^k(\Omt{1}, \Omt{2}) : v |_{\partial\Om{0}} = 0\},
%\end{split}
%\label{defbrokenSobolevspace}
%\end{equation}
%%
%where $H^k$ denotes the Sobolev space $W^{k,2}$.
%
From the simple discontinuous evolution of $\mathcal{T}_G$, we have via (\ref{slabwise_s_domains}) that
\begin{equation}
A_{n} = A_{h,t_n} = A_{h,t} \quad \forall t \in I_n
\label{Ahndef}
\end{equation}
%
%Note that we have 
%%
%\begin{equation}
%\int_{\bGn} wv  \ud \bs = \int_{I_n} (w, v)_{\Gamma(t)} \ud t. \label{Gipint}
%\end{equation}  
%
Let $\Gamma_K(t) := K \cap \Gamma(t)$. We define the following two mesh-dependent norms:
\begin{equation}
\| w \|_{1/2,h,\Gamma(t)}^2 := \sum_{K \in \mathcal{T}_{0,\Gamma(t)}} h_K^{-1} \| w \|_{\Gamma_K(t)}^2 \quad \quad \| w \|_{-1/2,h,\Gamma(t)}^2 := \sum_{K \in \mathcal{T}_{0,\Gamma(t)}} h_K \| w \|_{\Gamma_K(t)}^2
\label{def:HHnorm} 
\end{equation}
Note that
\begin{equation}
\| w \|_{\Gamma(t)}^2 \leq h \| w \|_{1/2,h,\Gamma(t)}^2 \quad \quad (w,v)_{\Gamma(t)} \leq \| w \|_{-1/2,h,\Gamma(t)}\| v \|_{1/2,h,\Gamma(t)}
\label{HHnormineqs}
\end{equation}
%
%\begin{equation}
%\| w \|_{\Gamma(t)}^2 \leq h \| w \|_{1/2,h,\Gamma(t)}^2,
%\label{HHnorm2ournormineq}
%\end{equation}
%%
%and 
%%
%\begin{equation}
%(w,v)_{\Gamma(t)} \leq \| w \|_{-1/2,h,\Gamma(t)}\| v \|_{1/2,h,\Gamma(t)}.
%\label{HHnormCS}
%\end{equation}
%
We define the time-dependent spatial energy norm $\norma{\cdot}$ by
\begin{equation}
\begin{split}
\norma{w}^2 := \sum_{i = 1}^2 \| \nab w \|_{\Omt{i}}^2 + \|\langle \partial_{n} w \rangle \|_{-1/2,h,\Gamma(t)}^2 + \|[w] \|_{1/2,h,\Gamma(t)}^2 + \|[\nab w]\|_{\Omt{O}}^2
\end{split} 
\label{def:anorm}   
\end{equation}   
Continuity of $A_{h,t}$ follows from using \eqref{HHnormineqs} in (\ref{Ahdef}). Next we consider the coercivity:
 
\begin{lemma}[Discrete coercivity of $A_{h,t}$] \label{Ahtcoerlem} 
Let the bilinear form $A_{h,t}$ and the energy norm $\norma{\cdot}$ be defined by \eqref{Ahdef} and \eqref{def:anorm}, respectively. Then, for $t \in [0, T]$ and $\gamma$ sufficiently large,
\begin{equation}
A_{h,t}(v,v) \gtrsim \norma{v}^2 \quad \forall v \in V_h(t) \label{Ahtcoer}
\end{equation}
\begin{proof} Following the proof of the coercivity in \cite{Hansbo:2002aa}, we consider
\begin{equation}
\begin{split}
2(\langle \partial_{n} v \rangle, [v])_{\Gamma(t)} \le & \; \frac{1}{\varepsilon}\|\langle \partial_{n} v \rangle \|_{-1/2,h,\Gamma(t)}^2 + \varepsilon \| [v] \|_{1/2,h,\Gamma(t)}^2 \\
\le & \; \frac{2}{\varepsilon} C_I \bigg( \sum_{i=1}^2 \| \nab v \|_{\Omt{i}}^2 + \| [\nab v]\|_{\Omt{O}}^2 \bigg) \\
& - \frac{1}{\varepsilon}\|\langle \partial_{n} v \rangle \|_{-1/2,h,\Gamma(t)}^2 + \varepsilon\| [v] \|_{1/2,h,\Gamma(t)}^2
\end{split} \label{Ahtcoerlemvvmid}
\end{equation}
where we have used Lemma~\ref{invineqgamlem} and denoted its constant by $C_I$ . We use \eqref{Ahtcoerlemvvmid} in
\begin{equation}
\begin{split}
A_{h,t}(v,v) = & \sum_{i=1}^2 \|\nab v \|_{\Omt{i}}^2 - 2(\langle \partial_{n} v \rangle, [v])_{\Gamma(t)}
+ \gamma \| [v]\|_{1/2,h,\Gamma(t)}^2 + \|[\nab v]\|_{\Omt{O}}^2 \\
\ge & \; \bigg(1 - \frac{2 C_I}{\varepsilon}\bigg) \sum_{i=1}^2\|\nab v \|_{\Omt{i}}^2 + \frac{1}{\varepsilon}\|\langle \partial_{n} v \rangle \|_{-1/2,h,\Gamma(t)}^2 \\
& + (\gamma - \varepsilon ) \| [v] \|_{1/2,h,\Gamma(t)}^2 + \bigg(1 - \frac{2 C_I}{\varepsilon} \bigg) \|[\nab v]\|_{\Omt{O}}^2 
\end{split} \label{Ahtcoerlemvvfin}
\end{equation}
By taking $\varepsilon > 2C_I$, and $\gamma > \varepsilon$ we may obtain (\ref{Ahtcoer}) from (\ref{Ahtcoerlemvvfin}).
\end{proof}
\end{lemma}

\subsection{Standard operators that map to $V_h(t)$}

Here we define some standard spatial operators for every $t \in [0, T]$. The $L^2(\Om{0})$-projection operator $P_{h,t} : L^2(\Om{0}) \to V_h(t)$ is defined by
\begin{equation}
(P_{h,t}w, v)_{\Om{0}} = (w, v)_{\Om{0}} \quad \forall v \in V_h(t)
\label{Ptdef} 
\end{equation} 
The Ritz projection operator $R_{h,t} : H^{3/2 + \varepsilon}(\cup_i \Omt{i}) \to V_h(t)$ is defined by
%The Ritz projection operator $R_{h,t} : H^1_{\partial_n \Gt}(\cup_i \Omt{i}) \to V_h(t)$ is defined by
%
\begin{equation}
A_{h,t}(R_{h,t}w, v) = A_{h,t}(w, v) \quad \forall v \in V_h(t)
\label{Rtdef} 
\end{equation}
\begin{lemma}[Estimates for the Ritz projection error] \label{lem:ritzop_error}
Let the spatial energy norm $\norma{\cdot}$ and the Ritz projection operator $R_{h,t}$ be defined by \eqref{def:anorm} and \eqref{Rtdef}, respectively. Then for any $w \in H^{p+1}(\Om{0}) \cap H^1_0(\Om{0})$
\begin{align}
\norma{w - R_{h,t}w} & \lesssim h^p \| D_x^{p+1} w \|_{\Om{0}} \label{lemres:ritzop_error_energy} \\
\| w - R_{h,t}w \|_\Om{0} & \lesssim h^{p+1} \| D_x^{p+1} w \|_{\Om{0}} \label{lemres:ritzop_error}
\end{align}
\begin{proof}
The proof is essentially the same as in the standard case with only natural modifications to account for the CutFEM setting. First the energy estimate (\ref{lemres:ritzop_error_energy}) is shown, which then is used in the Aubin-Nitsche duality trick to show (\ref{lemres:ritzop_error}). Let $\psi =  w - R_{h,t}w$ denote the projection error. Due to only having \emph{discrete} coercivity of $A_{h,t}$, we use the spatial interpolation operator $I_{h, t}$ defined by (\ref{def:interph}) to split the error into $\pi = w - I_{h,t}w$ and $\eta = I_{h,t}w - R_{h,t}w$. Thus
\begin{equation}
\norma{w - R_{h,t}w} = \norma{\psi} \leq \norma{\pi} + \norma{\eta}
\label{ritzop_error_split_norm}
\end{equation}
Since $\eta \in V_h(t)$, we use Lemma~\ref{Ahtcoerlem} to get
\begin{equation}
\norma{\eta}^2 \lesssim A_{h,t}(\eta, \eta) = - A_{h,t}(\pi, \eta) \lesssim \norma{\pi} \norma{\eta}
\label{ritzop_error_split_eta}
\end{equation}
Using this in (\ref{ritzop_error_split_norm}) and applying Lemma~\ref{lem:interphest_energy} we get
\begin{equation}
\norma{w - R_{h, t} w} \lesssim \norma{\pi} = \norma{w - I_{h,t}w} \lesssim h^p \| D_x^{p+1} w\|_{\Om{0}}
\label{ritzop_error_energyfin}
\end{equation}
which is (\ref{lemres:ritzop_error_energy}). For \eqref{lemres:ritzop_error}, we consider the auxiliary problem: Find $\phi \in H^2(\Om{0}) \cap H_0^1(\Om{0})$ such that $\lap{\phi} = \psi \ \text{in} \ \Om{0}.$
%
%\begin{equation} 
%- \lap{\phi} = \psi \quad \text{in} \ \Om{0}.
%\label{ritzop_error_auxprob}
%\end{equation}
%
From regularity we have that $[\partial_n\phi]|_{\Gamma(t)} = 0$ in $L^2(\Gt)$. Using the integration by parts provided by Corollary~\ref{cor:partintbroksob_A} and the spatial interpolation operator $I_{h, t}$ for $p = 1$ in following the Aubin-Nitsche duality trick shows \eqref{lemres:ritzop_error}.
\end{proof}
\end{lemma}
\noindent The discrete Laplacian $\lap_{h,t} : H^{3/2 + \varepsilon}(\cup_i \Omt{i}) \to V_h(t)$ is defined by
%\noindent The discrete Laplacian $\lap_{h,t} : H^1_{\partial_n \Gt}(\cup_i \Omt{i}) \to V_h(t)$ is defined by
%
\begin{equation}
(-\lap_{h,t} w, v)_\Om{0} = A_{h,t}(w, v) \quad \forall v \in V_h(t)
\label{dlapdef} 
\end{equation}
From the simple discontinuous evolution of $\mathcal{T}_G$, we have via (\ref{slabwise_s_space}) and (\ref{Ahndef}) that
\begin{align}
P_{n} & = P_{h,t_n} = P_{h,t} \quad \forall t \in I_n \label{Pndef} \\
R_{n} & = R_{h,t_n} = R_{h,t} \quad \forall t \in I_n \label{Rndef} \\
\lap_{n} & = \lap_{h,t_n} = \lap_{h,t} \quad \forall t \in I_n \label{dlapndef}
\end{align}

\subsection{The shift operator}

Here, we introduce the shift operator which is not present in the standard analysis, presented in~\cite{Eriksson1991, Eriksson1995}. The shift operator is needed because of the simple discontinuous mesh evolution in the CutFEM setting. The shift operator will be used in the proof of Lemma~\ref{stablem2}. At one point in the proof, one would like to consider $R_{n}u_{h,n-1}^-$. This is however undefined in the current setting because of the shifting discontinuity coming from the evolution of $\Gamma$. Since $R_{n}$ is only defined for functions in $H^{3/2 + \varepsilon}(\cup_{i} \Om{{i,n}}) \subset H^1(\cup_{i} \Om{{i,n}})$
%$H^1_{\partial_n \Gn}(\cup_{i} \Om{{i,n}}) \subset H^1(\cup_{i} \Om{{i,n}})$ 
and $u_{h,n-1}^- \in V_{h,n-1} \subset H^1(\cup_{i} \Om{{i,n-1}})$, the projection $R_{n}u_{h,n-1}^-$ is not defined. Enter shift operator. The idea is to consider a Ritzlike operator that can map from one discrete space to another. To define the shift operator, we will use a special bilinear form $\mathscr{A}_{n, m}$. To define $\mathscr{A}_{n, m}$, we will use a partition of $\Om{0}$ into the subdomains
\begin{equation}
\om{{ij}} = \om{{i,n,j,m}} := \Om{{i, n}} \cap \Om{{j, m}} \quad \text{for } i, j = 1, 2 \text{ and } n,m = 0, \dots, N
\label{def:omegaij}
\end{equation}
For $n, m = 0, 1,  \dots, N$, we define the non-symmetric bilinear form $\mathscr{A}_{n, m}$ by
\begin{equation}
\begin{split}
\mathscr{A}_{n, m}(v, w) := \sum_{i, j = 1}^2 (\nab v, \nab w)_{\om{{ij}}} - ([v], \langle \partial_{n} w \rangle)_{\Gamma_{n}} - (\langle \partial_{n} v \rangle, [w])_{\Gamma_{m}} 
\end{split}
\label{def:specAh}
\end{equation}
We define a related energy norm by %on $H^1_0(\cup_i \Om{{i, n}})$ by
\begin{align}
\normaspecnm{v}^2 &:= \norman{v}^2 + \| \langle \partial_n v \rangle \|_{-1/2,h, \Gamma_{m}}^2 \label{def:normaspecnm}
\end{align}
With this norm, we may obtain the following continuity result:
\begin{lemma}[Continuity of $\mathscr{A}_{n, m}$]\label{lem:Aspeccont}
Let the bilinear form $\mathscr{A}_{n, m}$ be defined by \eqref{def:specAh}, and the norm $\normaspecnm{\cdot}$ by \eqref{def:normaspecnm}. Then for functions $v$ and $w$ of sufficient regularity
\begin{equation}
\mathscr{A}_{n, m}(v, w) \lesssim \normaspecnm{v} \normaspecmn{w} %\quad \forall v \in H^{3/2 + \varepsilon}(\cup_i \Om{{i, n}}) \text{ and } \forall w \in H^{3/2 + \varepsilon}(\cup_j \Om{{j, m}})
%\quad \forall v \in H^1_{\partial_n (\Gn \cup \Gm)}(\cup_i \Om{{i, n}}) \text{ and } \forall w \in H^1_{\partial_n (\Gm \cup \Gn)}(\cup_j \Om{{j, m}})
\label{lemres:Aspeccont}
\end{equation}

\begin{proof}
The left-hand side of \eqref{lemres:Aspeccont} is given by \eqref{def:specAh}, where the first term is
\begin{equation}
\sum_{i, j = 1}^2 (\nab v, \nab w)_{\om{{ij}}} \leq \norman{v} \normam{w} \leq \normaspecnm{v} \normaspecmn{w}
\label{Aspeccont0_1}
\end{equation}
Here we have used that $v \in H^1(\cup_i \Om{{i, n}})$ and $w \in H^1(\cup_j \Om{{j, m}})$ to merge integrals over $\om{{ij}}$'s to integrals over $\Om{i}$'s, e.g., $\| \nab v \|_{\om{{i1}}}^2 + \| \nab v \|_{\om{{i2}}}^2 = \| \nab v \|_{\Om{{i, n}}}^2$. For the second and third term, we use \eqref{HHnormineqs} followed by the norm definition \eqref{def:normaspecnm}.
\end{proof}
\end{lemma}

\noindent By restricting $v$ and $w$ in Lemma~\ref{lem:Aspeccont} to the corresponding discrete subspaces, i.e., $V_{h, n}$ and $V_{h, m}$, respectively, we may obtain a continuity result in the weaker $A_n$-norms. This is done by estimating the average terms in the $\mathscr{A}_{n,m}$-norms using an inverse inequality that is a twist on the one in Lemma~\ref{invineqgamlem}. The average term is
\begin{equation}
\| \langle \partial_{n} v \rangle \|_{-1/2,h,\Gamma_{m}}^2 \lesssim \norman{v}^2 + \| \langle \partial_{n} v \rangle \|_{-1/2,h,\Gamma_{m} \setminus \Gamma_{n}}^2
\label{Aspeccont0_211}
\end{equation}
where we also want to estimate the second term by $\norman{v}^2$. We do this by following the proof of Lemma~\ref{invineqgamlem}, omitting some of the steps that are the same. Partitioning $\Gamma_{m} \setminus \Gamma_{n}$ into $\grave{\Gamma}_{i} := (\Gamma_{m} \setminus \Gamma_{n}) \cap \Om{{i,n}}$, using the interdependent indices $i = 1,2$, and $j = 0, G$, and writing $\grave{\Gamma}_{i {K_j}} = K_j \cap \grave{\Gamma}_i$, we have for $v \in V_{h,n}$ that
\begin{equation}
\begin{split}
\| \langle \partial_{n} v \rangle \|_{-1/2,h,\Gamma_{m} \setminus \Gamma_{n}}^2 & \lesssim \sum_{\grave{\Gamma}_{i {K_j}}} \sum_{\sigma \in \{+, -\}} h_{K_j} \| (\nab v)_\sigma \|_{\grave{\Gamma}_{i {K_j}}}^2 \lesssim \sum_{\grave{\Gamma}_{i {K_j}}} \sum_{\sigma \in \{+, -\}} \|\nab v\|_{K_j^\sigma}^2 \\
& \lesssim \|\nab v\|_{\Om{{1,n}}}^2 + \|(\nab v)_1\|_{\Om{{O, n}}}^2 + \|\nab v\|_{\Om{{2, n}}}^2 \\
& \lesssim \sum_{i=1}^2 \|\nab v\|_{\Om{{i,n}}}^2 + \| [\nab v] \|_{\Om{{O,n}}}^2 \lesssim \norman{v}^2
\end{split} 
\label{Aspeccont0_2112}
\end{equation}
where we have used Corollary~\ref{cor:scatraineqGamK}. Using \eqref{Aspeccont0_2112} in \eqref{Aspeccont0_211}, we get for $v \in V_{h,n}$ that
\begin{equation}
\| \langle \partial_n v \rangle \|_{-1/2,h, \Gamma_{m}} \lesssim \norman{v}
\label{Aspecnormeqn_0}
\end{equation}
Using \eqref{Aspecnormeqn_0} in \eqref{def:normaspecnm}, we may obtain
\begin{align}
\norman{v} &\leq \normaspecnm{v} \lesssim \norman{v} \quad \forall v \in V_{h,n} \label{Aspecnormeqnm1}
\end{align}
By restricting the functions in Lemma~\ref{lem:Aspeccont} to the discrete subspaces, we may use the above norm equivalence to obtain the following discrete continuity result:
\begin{corollary}[Discrete continuity of $\mathscr{A}_{n, m}$]\label{cor:Aspecdisccont}
Let the bilinear form $\mathscr{A}_{n, m}$ and the spatial energy norm $\norman{\cdot}$ be defined by \eqref{def:specAh} and \eqref{def:anorm}, respectively. Then
\begin{equation}
\mathscr{A}_{n, m}(v, w) \lesssim \norman{v} \normam{w} \quad \forall v \in V_{h, n} \text{ and } \forall w \in  V_{h, m}
\label{corres:Aspecdisccont}
\end{equation}
\end{corollary}

\noindent We are now ready to move on to the shift operator.

\begin{definition}[Shift operator]\label{def:shiftop}
For $n, m = 0, \dots, N$, we define the shift operator $\mathscr{S}_{n,m} : H^{3/2 + \varepsilon}(\cup_i \Om{{i, n}}) \to V_{h, m}$ by
%For $n, m = 0, \dots, N$, we define the shift operator $\mathscr{S}_{n,m} : H^1_{\partial_n \Gm}(\cup_i \Om{{i, n}}) \to V_{h, m}$ by
% 
\begin{align}
A_{m}(\mathscr{S}_{n,m} v, w) &= \mathscr{A}_{n, m}(v, w) \quad \forall w \in V_{h, m} \label{eqdef:shiftop}
\end{align}
\end{definition}
\noindent Note that $V_{h,n} \subset H^{3/2 + \varepsilon}(\cup_i \Om{{i, n}})$
%$V_{h,n} \subset H^1_{\partial_n \Gm}(\cup_i \Om{{i, n}})$
so that $\mathscr{S}_{n,m} : V_{h,n} \to V_{h,m}$. When it is clear what type of function $\mathscr{S}_{n,m}$ operates on, we write $\mathscr{S}_m = \mathscr{S}_{n,m}$ for brevity. For $v \in H^{3/2 + \varepsilon}(\cup_i \Om{{i, n}})$
%$v \in H^1_{\partial_n (\Gn \cup \Gm)}(\cup_i \Om{{i, n}})$
, using $\mathscr{S}_m v \in V_{h,m}$, the discrete coercivity of $A_m$, the definition of $\mathscr{S}_m$, and the continuities of $\mathscr{A}_{n, m}$, we get the following stability of the shift operator:
\begin{equation}
\normam{\mathscr{S}_m v} \lesssim \normaspecnm{v} \quad \forall v \in H^{3/2 + \varepsilon}(\cup_i \Om{{i, n}})
%\normam{\mathscr{S}_m v} \lesssim \normaspecnm{v} \quad \forall v \in H^1_{\partial_n (\Gn \cup \Gm)}(\cup_i \Om{{i, n}})
\label{shiftop_stability_cont}
\end{equation}
By restricting $v$ in \eqref{shiftop_stability_cont} to $V_{h,n} \subset H^{3/2 + \varepsilon}(\cup_i \Om{{i, n}})$
%$V_{h,n} \subset H^1_{\partial_n (\Gn \cup \Gm)}(\cup_i \Om{{i, n}})$
and using \eqref{Aspecnormeqnm1}, we have 
\begin{equation}
\normam{\mathscr{S}_m v} \lesssim \norman{v} \quad \forall v \in V_{h,n}
\label{shiftop_stability}
\end{equation}
%
%\begin{definition}[Shift operator]\label{def:shiftop}
%For $n, m = 0, \dots, N$, we define the shift operator $\mathscr{S}_{n,m} : V_{h, n} \to V_{h, m}$ by
%% 
%\begin{align}
%A_{m}(\mathscr{S}_{n,m} v, w) &= \mathscr{A}_{n, m}(v, w) \quad \forall w \in V_{h, m} \label{eqdef:shiftop}
%\end{align}
%%
%\end{definition}
%%
%\noindent When it is clear what type of function $\mathscr{S}_{n,m}$ operates on, we write $\mathscr{S}_m = \mathscr{S}_{n,m}$ for brevity. For $v \in V_{h, n}$, using $\mathscr{S}_m v \in V_{h,m}$, the discrete coercivity of $A_m$, the definition of $\mathscr{S}_m$, and the discrete continuity of $\mathscr{A}_{n, m}$, we get the following stability of the shift operator:
%%
%\begin{equation}
%\normam{\mathscr{S}_m v} \lesssim \norman{v} \quad \forall v \in V_{h, n}
%\label{shiftop_stability}
%\end{equation}
%
The shift operator has two approximability properties that are essential for its application in the analysis. We present and prove these properties in the following two lemmas:
\begin{lemma}[An estimate for the shift error]\label{lem:shiftop_error}
Let the shift operator $\mathscr{S}_m = \mathscr{S}_{n,m}$ be defined by \eqref{eqdef:shiftop} and the spatial energy norms $\normaspecnm{\cdot}$ and $\norman{\cdot}$ by \eqref{def:normaspecnm} and \eqref{def:anorm}, respectively. Then
\begin{align}
\| v - \mathscr{S}_m v \|_\Om{0} & \lesssim h \normaspecnm{v} \quad \forall v \in H^{3/2 + \varepsilon}(\cup_i \Om{{i, n}})
%H^1_{\partial_n (\Gn \cup \Gm)}(\cup_i \Om{{i, n}})
\label{lemres:shiftop_error_cont} \\
\| v - \mathscr{S}_m v \|_\Om{0} & \lesssim h \norman{v} \quad \forall v \in V_{h,n}
\label{lemres:shiftop_error}
\end{align}
\begin{proof}
The estimate \eqref{lemres:shiftop_error} follows from \eqref{lemres:shiftop_error_cont} by restricting $v$ to $V_{h,n} \subset H^{3/2 + \varepsilon}(\cup_i \Om{{i, n}})$
%$V_{h,n} \subset H^1_{\partial_n (\Gn \cup \Gm)}(\cup_i \Om{{i, n}})$
and using \eqref{Aspecnormeqnm1}. We thus only need to prove \eqref{lemres:shiftop_error_cont}. The proof is based on the Aubin-Nitsche duality trick but involves a few modifications. Let $\psi = v - \mathscr{S}_m v$ denote the shift error. We consider the auxiliary problem: Find $\phi \in H^2(\Om{0}) \cap H_0^1(\Om{0})$ such that $-\lap{\phi} = \psi \ \text{in} \ \Om{0}$. We note from regularity that $[\partial_n\phi]|_{\Gamma_m \cup \Gamma_{n}} = 0$ in $L^2(\Gamma_m \cup \Gamma_{n})$. We denote by $I_{h, m} = I_{h, p=1, m}$ the spatial interpolation operator $I_{h, t_m}$ defined by (\ref{def:interph}). The square of the left-hand side of (\ref{lemres:shiftop_error}) is
\begin{equation}
\begin{split}
\| v - \mathscr{S}_m v \|_\Om{0}^2 & = (\psi, \psi)_{\Om{0}} = (v, -\lap \phi)_{\Om{0}} - (\mathscr{S}_m v, -\lap \phi)_{\Om{0}} \\
& = \sum_{i=1}^2 (\nab v, \nab \phi)_{\Om{{i,n}}} - ([v], \langle \partial_n \phi \rangle)_{\Gamma_{n}} - A_m(\mathscr{S}_m v, \phi) \\
& = \mathscr{A}_{n, m}(v, \phi) - A_m(\mathscr{S}_m v, \phi) \pm \mathscr{A}_{n, m}(v, I_{h, m} \phi) \\
& = \mathscr{A}_{n, m}(v, \phi - I_{h, m} \phi) - A_m(\mathscr{S}_m v, \phi - I_{h, m} \phi) \\
& \lesssim \normaspecnm{v} \normaspecmn{\phi - I_{h, m} \phi} + \normam{\mathscr{S}_m v}\normam{\phi - I_{h, m} \phi} \\
& \lesssim h \normaspecnm{v} \| D_x^2 \phi \|_{\Om{0}} \lesssim h \normaspecnm{v} \| \psi \|_{\Om{0}}
\end{split}
\label{shiftop_error_0}
\end{equation}
where we have used Lemma~\ref{lem:partintbroksob}, Corollary~\ref{cor:partintbroksob_A}, that $[\phi]|_{\Gamma_m} = 0$ to go to $\mathscr{A}_{n, m}$, the definition of the shift operator, the continuities, Lemma~\ref{lem:interphest_energy_special}, \eqref{shiftop_stability_cont}, Lemma~\ref{lem:interphest_energy}, and elliptic regularity on $H^2(\Om{0}) \cap H_0^1(\Om{0})$ for $\phi$. This shows \eqref{lemres:shiftop_error_cont}. 
\end{proof}
\end{lemma}

\begin{lemma}[An estimate for the shift energy]\label{lem:shiftenergy}
Let the bilinear form $A_n$ be defined by \eqref{Ahdef}, the shift operator $\mathscr{S}_m = \mathscr{S}_{n, m}$ by \eqref{eqdef:shiftop}, the spatial energy norm $\norman{\cdot}$ by \eqref{def:anorm}, and the discrete Laplacian $\lap_n$ by \eqref{dlapdef}. Then
\begin{equation}
A_{m}(\mathscr{S}_m v, \mathscr{S}_m v) - A_{n}(v, v) \lesssim h \norman{v} \| \lap_{n} v \|_{\Om{0}} \quad \forall v \in V_{h, n}
\label{lemres:shiftenergy}
\end{equation}
\begin{proof}
The left-hand side of \eqref{lemres:shiftenergy} is
\begin{equation}
\begin{split}
& \; A_{m}(\mathscr{S}_m v, \mathscr{S}_m v) - A_{n}(v, v) \\
= & \; A_{n}(v, \mathscr{S}_{n}\mathscr{S}_m v - v) = (\lap_{n} v, v - \mathscr{S}_{n} \mathscr{S}_m v)_{\Om{0}} \\
\leq & \; \| \lap_{n} v \|_{\Om{0}} \bigg(\| v - \mathscr{S}_m v \|_{\Om{0}} + \| \mathscr{S}_m v - \mathscr{S}_{n}\mathscr{S}_m v \|_{\Om{0}} \bigg)\\
\lesssim & \; \| \lap_{n} v \|_{\Om{0}} \bigg( h \norman{v} + h \normam{\mathscr{S}_m v} \bigg) \\
\lesssim & \; h \| \lap_{n} v \|_{\Om{0}} \norman{v}
\end{split}
\label{shiftenergy0}
\end{equation}
where we have used \eqref{eqdef:shiftop}, Lemma~\ref{lem:shiftop_error}, and \eqref{shiftop_stability}. This shows \eqref{lemres:shiftenergy}.
\end{proof}
\end{lemma}

\subsection{The bilinear form $B_h$}
The bilinear form $B_h$ can be expressed differently, as noted in the following lemma: 

\begin{lemma}[Alternative form of $B_h$] \label{lem:Bhpit} 
The bilinear form $B_h$, defined by \eqref{Bhdef}, can be written as
\begin{equation}
\begin{split}
B_h(w, v) = \sum_{n=1}^N \int_{I_n} (w,- \dot v)_{\Om{0}} + A_{h,t}(w, v) \ud t + \sum_{n=1}^{N-1}(w_{n}^-,-[v]_{n})_{\Om{0}} + (w_{N}^-, v_{N}^-)_{\Om{0}}  
\end{split} \label{lemres:Bhpit}
\end{equation}

\begin{proof}
The proof is exactly as in the standard case. The first term in (\ref{Bhdef}) is integrated by parts with respect to time and the result is combined with the last two terms in (\ref{Bhdef}).
\end{proof}
\end{lemma}
\noindent To show Galerkin orthogonality, we need the following lemma on consistency:
\begin{lemma}[Consistency] \label{cnstylem}  
The solution $u$ to problem \eqref{heateq} also solves \eqref{feform}.
\begin{proof}
First insert $u$ in place of $u_h$ on the left-hand side of \eqref{feform} and use the regularity of $u$. Then integrate by parts in space to get interior and boundary terms. The exterior boundary terms vanish because of the boundary conditions imposed on $v$ thus leaving the $\Gamma$-terms which are combined. Applying Lemma~\ref{jmplem} and the regularity of $u$ only leaves terms which from \eqref{heateq} equals the right-hand side of \eqref{feform}.
\end{proof}
\end{lemma}

\noindent From Lemma~\ref{cnstylem} we may obtain the Galerkin orthogonality:
\begin{corollary}[Galerkin orthogonality]\label{cor:galort}
Let the bilinear form $B_h$ be defined by \eqref{Bhdef}, and let $u$ and $u_h$ be the solutions of \eqref{heateq} and \eqref{feform}, respectively. Then
\begin{equation}
B_h(u - u_h, v) = 0 \quad \forall v \in V_h \label{galort}
\end{equation}
\end{corollary} 
\noindent We now consider the discrete dual problem: Find $z_h \in V_h$ such that
\begin{equation}
B_h(v, z_h) = (v_N^-, z_{h,N}^+)_{\Om{0}} \quad \forall v \in V_h \label{discdp}
\end{equation}

\section{Stability analysis} \label{secstaban}

The stability analysis in this section is based on a stability analysis for the case with only a background mesh, presented by Eriksson and Johnson in~\cite{Eriksson1991, Eriksson1995}. Due to the CutFEM setting, the original analysis has been slightly modified by the incorporation of the shift operator defined by \eqref{eqdef:shiftop}. The main result of this section is the following stability estimate and its counterpart for the discrete dual problem:

\begin{theorem}[The main stability estimate] \label{stabestmainthm}

Let $u_h$ be the solution of \eqref{feform} with $f \equiv 0$ and let $u_0$ be the initial value of the analytic solution of the problem presented in Section \ref{sec:probform}. Then we have that
\begin{equation}
\| u_{h,N}^- \|_{\Om{0}} + \sum_{n=1}^N \int_{I_n} \| \dot{u}_h \|_{\Om{0}} + \| \lap_{n} u_h \|_{\Om{0}} \ud t + \sum_{n=1}^N \| [u_h]_{n-1} \|_{\Om{0}} \leq C_1 \| u_0 \|_{\Om{0}}
\label{stabestmain}
\end{equation}
where $C_1 = C(\log(t_N/k_1) + 1)^{1/2}$ and $C > 0$ is a constant.

\end{theorem}

\noindent The counterpart of (\ref{stabestmain}) for $z_h$ is a crucial tool in the proof of the a priori error estimate presented in Theorem~\ref{aprithm}. For the purpose of that application, we replace the initial time jump term. From (\ref{stabestmain}), we have that $\| [u_h]_{0} \|_{\Om{0}} \leq C_1 \| u_0 \|_{\Om{0}}$. The corresponding inequality for $z_h$ is $\| [z_h]_{N} \|_{\Om{0}} \leq C_N \| z_{h,N}^+ \|_{\Om{0}}$, where $C_N = C(\log(t_N/k_N) + 1)^{1/2}$ and $C > 0$. By squaring both sides of this inequality and expanding the left-hand side we may obtain $\| z_{h,N}^- \|_{\Om{0}} \leq C_N \| z_{h,N}^+ \|_{\Om{0}}$,
%
%\begin{equation}
%\begin{split}
%C_N^2 \| z_{h,N}^+ \|_{\Om{0}}^2 & \geq \| [z_h]_{N} \|_{\Om{0}}^2 \geq \| z_{h,N}^- \|_{\Om{0}}^2 - 2(z_{h,N}^+, z_{h,N}^-)_{\Om{0}},
%\end{split}
%\label{stabestmain_diff_tj_zhN_0}
%\end{equation}
%
%from which we get
%
%\begin{equation}
% \| z_{h,N}^- \|_{\Om{0}} \leq C_N \| z_{h,N}^+ \|_{\Om{0}}.
%\label{stabestmain_diff_tj_zhN_1}
%\end{equation}
%
which we use in the corresponding stability estimate for $z_h$:

\begin{corollary}[A stability estimate for $z_h$]  \label{stabestmainzhcor}

A corresponding stability estimate to \eqref{stabestmain} for the solution $z_h$ to the discrete dual problem \eqref{discdp} is 
\begin{equation}
\| z_{h,0}^+ \|_{\Om{0}} + \sum_{n=1}^N \int_{I_n} \| \dot{z}_h \|_{\Om{0}} + \| \lap_{n} z_h \|_{\Om{0}} \ud t + \sum_{n=1}^{N-1} \| [z_h]_{n} \|_{\Om{0}} + \| z_{h,N}^- \|_{\Om{0}} \leq C_N \| z_{h,N}^+ \|_{\Om{0}}
\label{stabestmainzh}
\end{equation}
where $C_N = C(\log(t_N/k_N) + 1)^{1/2}$ and $C > 0$ is a constant.

\end{corollary}

\noindent To prove Theorem~\ref{stabestmainthm}, we need two other stability estimates. We start by letting $f \equiv 0$ in (\ref{feform}). We have: Find $u_h \in V_h$ such that 
\begin{equation}
\sum_{n=1}^N \int_{I_n} (\dot{u}_h, v)_{\Om{0}} \ud t + \sum_{n=1}^N \int_{I_n} A_{n}(u_h, v) \ud t + \sum_{n=1}^N ([u_h]_{n-1}, v_{n-1}^+)_{\Om{0}} = 0 \quad \forall v \in V_h      
\label{feformf0}
\end{equation}
The first of the two auxiliary stability estimates is:

\begin{lemma}[The basic stability estimate] \label{stablem1}
Let $u_h$ be the solution of \eqref{feform} with $f \equiv 0$ and let $u_0$ be the initial value of the exact solution $u$. Then
\begin{equation}
 \| u_{h,N}^- \|_{\Om{0}}^2 + \sum_{n=1}^N \int_{I_n} \norman{u_h}^2 \ud t + \sum_{n = 1}^N \| [u_h]_{n-1} \|_{\Om{0}}^2 \lesssim \| u_0 \|_{\Om{0}}^2
\label{stablem1res}
\end{equation}
\begin{proof}
The proof follows the same idea as in the standard case. By taking $v = u_h \in V_h$ in \eqref{feformf0}, integrating the first term and combining the result with the third term, estimate \eqref{stablem1res} follows after using Lemma~\ref{Ahtcoerlem} on the second term.

\end{proof}

\end{lemma}

\begin{lemma}[The strong stability estimate] \label{stablem2}
Let $u_h$ be the solution of \eqref{feform} with $f \equiv 0$ and let $u_0$ be the initial value of the exact solution $u$. Then
\begin{equation}
\sum_{n=1}^N t_n \int_{I_n} \| \dot{u}_h \|_{\Om{0}}^2 + \| \lap_{n} u_h \|_{\Om{0}}^2\ud t + \sum_{n=2}^N \frac{t_n}{k_n} \| [u_h]_{n-1} \|_{\Om{0}}^2 \lesssim \| u_0 \|_{\Om{0}}^2
\label{stablem2res}
\end{equation}

\begin{proof}
%
% ALTERNATIVE APPROACH (DOESNT WORK WITH ONLY t_1) 
By taking $v = -t_n\lap_{n}u_h \in V_h$ in (\ref{feformf0}) and using \eqref{dlapdef}, we may obtain
\begin{equation}
\begin{split}
& \sum_{n=1}^N t_n\int_{I_n} \| \lap_{n}u_h \|_\Om{0}^2 \ud t + \frac{t_1}{2}A_{1}(u_{h,0}^+, u_{h,0}^+) + \frac{t_N}{2}A_{N}(u_{h,N}^-, u_{h,N}^-) \\
 = & \; t_1(u_{h,0}^-, -\lap_{1}u_{h,0}^+)_{\Om{0}} \\
 & + \sum_{n=1}^{N-1}\bigg(t_{n+1}(u_{h,n}^-, -\lap_{n+1}u_{h,n}^+)_{\Om{0}} - \frac{t_{n+1}}{2}A_{n+1}(u_{h,n}^+, u_{h,n}^+) - \frac{t_n}{2}A_{n}(u_{h,n}^-, u_{h,n}^-) \bigg) \\
\end{split}
\label{stablem2dlapmid}
\end{equation}
%
%\begin{equation}
%\int_{I_n} \| \lap_{n}u_h \|_\Om{0}^2 \ud t + \frac{1}{2}A_{n}(u_{h,n}^-, u_{h,n}^-) + \frac{1}{2}A_{n}(u_{h,n-1}^+, u_{h,n-1}^+) - (u_{h,n-1}^-, -\lap_{n}u_{h,n-1}^+)_{\Om{0}} = 0
%\label{stablem2dlapmid}
%\end{equation}
%
The first term on the left-hand side is done. Using Lemma~\ref{Ahtcoerlem}, the other two terms on the left-hand side are estimated from below by $0$. Using standard estimates, the first term on the right-hand side is
\begin{equation}
t_1(u_{h,0}^-, -\lap_{1}u_{h,0}^+)_{\Om{0}} \lesssim \frac{1}{\varepsilon}\| u_{0} \|_{\Om{0}}^2 + \varepsilon t_1 \int_{I_1} \| \lap_{1}u_{h}\|_{\Om{0}}^2 \ud t
\label{stablem2dlapmid_rhsn1}
\end{equation}
We move on to the last row of \eqref{stablem2dlapmid}. We would like to use \eqref{dlapdef} on the first term, but due to the simple discontinuous evolution of $\mathcal{T}_G$, $u_{h,n}^- \notin V_{h,n+1}$. To handle this, we make use of the shift operator $\mathscr{S}_{n+1}: V_{h, n} \to V_{h, n+1}$ defined by \eqref{eqdef:shiftop}:
\begin{equation}
\begin{split}
& \; t_{n+1}(u_{h,n}^-, -\lap_{n+1}u_{h,n}^+)_{\Om{0}} \\
= & \; t_{n+1}((\mathds{1} - \mathscr{S}_{n+1})u_{h,n}^-, -\lap_{n+1}u_{h,n}^+)_{\Om{0}} + t_{n+1}(\mathscr{S}_{n+1}u_{h,n}^-, -\lap_{n+1}u_{h,n}^+)_{\Om{0}}
\end{split}   
\label{stablem2dlapmid_rhsn2N}
\end{equation}
Using standard estimates, Lemma~\ref{lem:shiftop_error}, and \eqref{quasiuniformity_st}, the first term is
\begin{equation}
t_{n+1}((\mathds{1} - \mathscr{S}_{n+1})u_{h,n}^-, -\lap_{n+1}u_{h,n}^+)_{\Om{0}} \lesssim \frac{1}{\varepsilon}\int_{I_{n}} \norman{u_{h}}^2 \ud t + \varepsilon t_{n+1}\int_{I_{n+1}} \| \lap_{n+1}u_{h}\|_{\Om{0}}^2 \ud t
\label{stablem2dlapmid_rhsn2N_1}
\end{equation}
Using \eqref{dlapdef} on the second term in (\ref{stablem2dlapmid_rhsn2N}) and combining it with the other two terms in the last row of \eqref{stablem2dlapmid} we have
\begin{equation}
\begin{split}
& \; t_{n+1}A_{n+1}(\mathscr{S}_{n+1}u_{h,n}^-, u_{h,n}^+) - \frac{t_{n+1}}{2}A_{n+1}(u_{h,n}^+, u_{h,n}^+) - \frac{t_{n}}{2}A_{n}(u_{h,n}^-, u_{h,n}^-) \\
\leq & \; \frac{t_{n+1}}{2}A_{n+1}(\mathscr{S}_{n+1}u_{h,n}^-, \mathscr{S}_{n+1}u_{h,n}^-) - \frac{t_{n}}{2}A_{n}(u_{h,n}^-, u_{h,n}^-) \\
\lesssim & \; \bigg(1 + \frac{1}{\varepsilon}\bigg) \int_{I_{n}} \norman{u_{h}}^2 \ud t + \varepsilon t_n \int_{I_{n}} \| \lap_{n}u_{h}\|_{\Om{0}}^2 \ud t   
\end{split}
\label{stablem2dlapmid_rhsn2N2_21}
\end{equation}
where we have first used the algebraic identity $(a - b)^2 = a^2 - 2ab + b^2$ and Lemma~\ref{Ahtcoerlem} to estimate the difference term from below by $0$, then Lemma~\ref{lem:shiftenergy}, \eqref{quasiuniformity_st}, that $t_{n+1} \lesssim t_{n}$ for $n = 1, \dots, N-1$, which follows from the temporal quasi-uniformity, and standard estimates. Collecting the estimates, we have
\begin{equation}
\begin{split}
& \sum_{n=1}^N t_n \int_{I_n} \| \lap_{n}u_h \|_\Om{0}^2 \ud t \\
\lesssim & \; \frac{1}{\varepsilon} \| u_{0} \|_{\Om{0}}^2 + \frac{1}{\varepsilon} \sum_{n=1}^{N-1} \int_{I_{n}} \norman{u_{h}}^2 \ud t + \varepsilon \sum_{n=1}^{N} t_n \int_{I_{n}} \| \lap_{n}u_{h}\|_{\Om{0}}^2 \ud t
\end{split}
\label{stablem2dlapfin}
\end{equation}
Kicking back the last term on the right-hand side, taking $\varepsilon$ sufficiently small, and using Lemma~\ref{stablem1}, we may obtain
\begin{equation}
\sum_{n=1}^N t_n \int_{I_n} \| \lap_{n}u_h \|_\Om{0}^2 \ud t \lesssim \| u_{0} \|_{\Om{0}}^2
\label{stablem2_dlapest}
\end{equation}
It is thus sufficient to estimate the time-derivative terms and the time jump terms on the left-hand side of (\ref{stablem2res}) by the left-hand side of (\ref{stablem2_dlapest}) to obtain (\ref{stablem2res}).
We proceed by taking $v = (t - t_{n-1})\dot{u}_h$ in \eqref{feformf0}. The subsequent treatment is exactly the same as in the standard case, and yields
%Since $((t - t_{n-1})\dot{u}_h)_{n-1}^+ = 0$ we have
%%
%\begin{equation}
%\begin{split}
%0 & = \int_{I_n} (\dot{u}_h, (t - t_{n-1})\dot{u}_h)_{\Om{0}} \ud t + \int_{I_n} A_{n}(u_h, (t - t_{n-1})\dot{u}_h) \ud t \\
%& = \int_{I_n} (t - t_{n-1}) \| \dot{u}_h \|_{\Om{0}}^2 \ud t + \int_{I_n} (t - t_{n-1})(- \lap_n u_h, \dot{u}_h)_{\Om{0}} \ud t.
%\end{split}
%\label{stablem2dtu0}
%\end{equation}
%%
%Using the Cauchy--Schwarz inequality twice, we may obtain
%%
%\begin{equation}
%\int_{I_n} (t - t_{n-1}) \| \dot{u}_h \|_{\Om{0}}^2 \ud t \leq \int_{I_n} (t - t_{n-1})\| \lap_n u_h\|_{\Om{0}}^2 \ud t. 
%\label{stablem2dtu_mid}
%\end{equation}
%%
%By using an inverse estimate and \eqref{stablem2dtu_mid}, we have
%
%\begin{equation}
%\int_{I_n} \| \dot{u}_h \|_{\Om{0}}^2 \ud t \leq Ck_n^{-1} \int_{I_n} (t - t_{n-1}) \| \dot{u}_h \|_{\Om{0}}^2 \ud t \leq C \int_{I_n} \| \lap_n u_h\|_{\Om{0}}^2 \ud t.
%\label{stablem2dtu_fin}
%\end{equation}
%
\begin{equation}
\int_{I_n} \| \dot{u}_h \|_{\Om{0}}^2 \ud t \lesssim \int_{I_n} \| \lap_n u_h\|_{\Om{0}}^2 \ud t
\label{stablem2dtu_fin}
\end{equation}
We proceed by showing an estimate for the time jump terms for $n = 2, \dots, N$. We would like to take $v = [u_h]_{n-1} = u_{h, n-1}^+ - u_{h, n-1}^-$ in (\ref{feformf0}), but due to the simple discontinuous evolution of $\mathcal{T}_G$, $u_{h,n-1}^- \notin V_{h,n}$ as already pointed out, so we cannot make this choice. To handle this, we use the $L^2(\Om{0})$-projection $P_n$. By taking $v = P_n[u_h]_{n-1}$ in \eqref{feformf0} and using \eqref{Ptdef} and \eqref{dlapdef}, we may obtain
%
%\begin{equation}
%\begin{split}
%& \; \int_{I_n} (\dot{u}_h, P_n[u_h]_{n-1})_{\Om{0}} \ud t + \int_{I_n} A_{n}(u_h, P_n[u_h]_{n-1}) \ud t + ([u_h]_{n-1}, (P_n[u_h]_{n-1})_{n-1}^+)_{\Om{0}} \\
%= & \; 0.
%\end{split}   
%\label{stablem2Ptju0}
%\end{equation}
%
%Using \eqref{Ptdef} and \eqref{dlapdef}, we may obtain
%We consider the terms separately, starting with the first:
%%
%\begin{equation}
%\int_{I_n} (\dot{u}_h, P_n[u_h]_{n-1})_{\Om{0}} \ud t = \int_{I_n} (\dot{u}_h, [u_h]_{n-1})_{\Om{0}} \ud t.
%\label{stablem2Ptju0_1}
%\end{equation}
%%
%The second term in \eqref{stablem2Ptju0} is
%%
%\begin{equation}
%\int_{I_n} A_{n}(u_h, P_n[u_h]_{n-1}) \ud t = \int_{I_n} (- \lap_n u_h, [u_h]_{n-1})_{\Om{0}} \ud t.
%\label{stablem2Ptju0_2}
%\end{equation}
%%
%The third term in \eqref{stablem2Ptju0} is
%%
%\begin{equation}
%([u_h]_{n-1}, (P_n[u_h]_{n-1})_{n-1}^+)_{\Om{0}} = \| [u_h]_{n-1} \|_{\Om{0}}^2 - ([u_h]_{n-1}, (\mathds{1} - P_n)[u_h]_{n-1})_{\Om{0}}.
%\label{stablem2Ptju0_3}
%\end{equation}
%%
%Using the identities (\ref{stablem2Ptju0_1}), (\ref{stablem2Ptju0_2}), and (\ref{stablem2Ptju0_3}) in (\ref{stablem2Ptju0}), we may obtain
%
\begin{equation}
 \| [u_h]_{n-1} \|_{\Om{0}}^2 = ([u_h]_{n-1}, (\mathds{1} - P_n)[u_h]_{n-1})_{\Om{0}} + \int_{I_n} (\lap_n u_h - \dot{u}_h, [u_h]_{n-1})_{\Om{0}} \ud t
\label{stablem2Ptju_mid}
\end{equation}
We treat the terms separately, starting with the first. For $n = 2, \dots, N$, we use standard estimates, Lemma~\ref{lem:shiftop_error}, and the spatiotemporal quasi-uniformity \eqref{quasiuniformity_st} to get
\begin{equation}
([u_h]_{n-1}, (\mathds{1} - P_n)[u_h]_{n-1})_{\Om{0}} \leq \frac{1}{4}\| [u_h]_{n-1}\|_{\Om{0}}^2 + C k_n \int_{I_{n-1}} \normanm{u_h}^2 \ud t
\label{stablem2Ptju_mid_rhs1}
\end{equation}
Using standard estimates, the second term in (\ref{stablem2Ptju_mid}) is 
\begin{equation}
\int_{I_n} (\lap_n u_h - \dot{u}_h, [u_h]_{n-1})_{\Om{0}} \ud t \leq \frac{1}{4}\|[u_h]_{n-1}\|_{\Om{0}}^2 + 2 k_n \int_{I_n} \| \lap_n u_h \|_{\Om{0}}^2 + \| \dot{u}_h\|_{\Om{0}}^2 \ud t
\label{stablem2Ptju_mid_rhs2}
\end{equation}
Using the estimates (\ref{stablem2Ptju_mid_rhs1}) and (\ref{stablem2Ptju_mid_rhs2}) in (\ref{stablem2Ptju_mid}), we may obtain, for $n = 2, \dots, N$, that
\begin{equation}
\frac{1}{k_n} \| [u_h]_{n-1} \|_{\Om{0}}^2 \lesssim \int_{I_n} \| \dot{u}_h\|_{\Om{0}}^2 + \| \lap_n u_h \|_{\Om{0}}^2 \ud t + \int_{I_{n-1}} \normanm{u_h}^2 \ud t
\label{stablem2Ptju_fin}
\end{equation}
Finally we have all the partial results that are needed to show the desired stability estimate \eqref{stablem2res}. Starting with the left-hand side of \eqref{stablem2res}, first using \eqref{stablem2Ptju_fin}, then applying Lemma~\ref{stablem1} followed by \eqref{stablem2dtu_fin}, and finally using \eqref{stablem2_dlapest} concludes the proof of Lemma~\ref{stablem2}.

\end{proof}

\end{lemma}

%\noindent Theorem~\ref{stabestmainthm} may now be proven by deriving lower bounds with the Cauchy--Schwarz inequality for separate terms on the left-hand sides of the auxiliary stability estimates.

\noindent Theorem~\ref{stabestmainthm} may now be proven by deriving lower bounds for the left-hand sides of the auxiliary stability estimates. To do this, we use the Cauchy--Schwarz inequality to obtain
%
%\begin{equation}
%\bigg( \sum_{n = 1}^N a_n b_n \bigg)^2 \bigg( \sum_{n = 1}^N a_n^2\bigg)^{-1} \leq \sum_{n = 1}^N b_n^2 \label{anbnineq}
%\end{equation}
\begin{equation}
\bigg( \sum_{n = 1}^N \| w \| \bigg)^2 \bigg( \sum_{n = 1}^N \frac{k_n}{t_n} \bigg)^{-1} \leq \sum_{n = 1}^N \frac{t_n}{k_n} \| w \|^2 \label{anbnineq}
\end{equation}
where $\| w\|$ denotes a generic norm term. We also use that
%
%\begin{equation}
%\sum_{n = 1}^N \frac{k_n}{t_n} = 1 + \sum_{n = 2}^N \frac{k_n}{t_n} \le 1 + \int_{t_1}^{t_N} \frac{1}{t} \ud t = 1 + \log(t_N/k_1) \label{lnineq}
%\end{equation} 
\begin{equation}
\sum_{n = 1}^N \frac{k_n}{t_n} \leq 1 + \log(t_N/k_1) \label{lnineq}
\end{equation}

\section{Error analysis} \label{secerran}

To prove an a priori error estimate, we follow the methodology presented by Eriksson and Johnson in~\cite{Eriksson1991, Eriksson1995} and make only minor modifications to account for the CutFEM setting.

\begin{theorem}[An optimal order a priori error estimate in $\| \cdot \|_{\Om{0}}$ at the final time] \label{aprithm}
Let $u$ be the solution of \eqref{heateq} and let $u_h$ be the finite element solution defined by \eqref{feform}. Then, for $q = 0, 1$, we have that
\begin{equation}
\| u(t_N) - u_{h,N}^- \|_{\Om{0}} \leq C_N \max_{1 \le n \le N} \bigg\{k_n^{2q+1}\| \dot{u}^{(2q+1)} \|_{\Om{0}, I_n} + h^{p+1} \| D_x^{p+1} u \|_{\Om{0}, I_n} \bigg\}
\label{apri}
\end{equation}  
where $\| \cdot \|_{\Om{0}} = \| \cdot \|_{L^2(\Om{0})}$, $C_N = C(\log(t_N/k_N) + 1)^{1/2}$, where $C > 0$ is a constant, $k_n = t_n - t_{n-1}$, $\| w \|_{\Om{0}, I_n} = \max_{t \in I_n} \|w \|_{\Om{0}}$, $\dot{u}^{(2q+1)} = \partial^{2q+1} u / \partial t^{2q+1}$, $h$ is the largest diameter of a simplex in $\mathcal{T}_0 \cup \mathcal{T}_G$, and $D_x$ denotes the derivative with respect to space.
\end{theorem}
\begin{proof}
We use the interpolant $\tilde{u} = \tilde{I}^n R_{n} u \in V_h$, where $\tilde{I}^n$ is the temporal interpolation operator defined by (\ref{ipdef}) and $R_n$ is the Ritz projection operator defined by (\ref{Rtdef}), to split the error $e = u - u_h$ into $\rho = u - \tilde{u}$ and $\theta = \tilde{u} - u_h$. Thus
\begin{equation}
\| u(t_N) - u_{h,N}^- \|_{\Om{0}} = \|e_N^-\|_{\Om{0}} \le \|\rho_N^-\|_{\Om{0}} + \|\theta_N^-\|_{\Om{0}}
\label{errorsplit_norm} 
\end{equation}
For $\rho$, we have from \eqref{ipdefnm} that for $n = 1, \dots, N$,
\begin{equation}
\begin{split}
\| \rho_n^- \|_{\Om{0}} & = \|(u - \tilde{I}^nR_{n}u)_n^- \|_{\Om{0}} = \| u_n^- - R_{n}u_n^- \|_{\Om{0}} \leq \| u - R_{n}u \|_{\Om{0}, I_n}
\end{split}
\label{rhonm}
\end{equation} 
For $\theta$, we note that since $z_h \in V_h$, the Galerkin orthogonality (\ref{galort}) gives
\begin{equation}
B_h(\theta, z_h) = - B_h(\rho, z_h)
\label{thetagalort}
\end{equation} 
Since $\theta \in V_h$, the discrete dual problem \eqref{discdp} with $z_{h,N}^+ = \theta_N^-$ gives
\begin{equation}
B_h(\theta, z_h) = \|\theta_N^-\|_{\Om{0}}^2
\label{thetazdef}
\end{equation} 
Combining (\ref{thetagalort}) with (\ref{thetazdef}), and using Lemma~\ref{lem:Bhpit}, we obtain the error representation
\begin{equation}
\begin{split}
\|\theta_N^-\|_{\Om{0}}^2 = & \sum_{n=1}^N \int_{I_n} (\rho, \dot z_h)_{\Om{0}} - A_{h,t}(\rho, z_h) \ud t + \sum_{n=1}^{N-1} (\rho_{n}^-, [z_h]_{n})_{\Om{0}} - (\rho_{N}^-, z_{h,N}^-)_{\Om{0}}
\end{split}
\label{thetaetaZdp}
\end{equation} 
Treating the terms on the right-hand side just as in~\cite{Eriksson1991, Eriksson1995}, we obtain for $q = 0$
\begin{equation}
\begin{split}
\|\theta_N^-\|_{\Om{0}}^2 \le & \; \max_{1 \le n \le N} \bigg\{\|u - R_{n}u \|_{\Om{0}, I_n} + \| R_{n}u - \tilde{u} \|_{\Om{0}, I_n} \bigg\} \\
& \times \bigg(\sum_{n=1}^N \int_{I_n} \| \lap_{n} z_h \|_{\Om{0}} \ud t + \sum_{n=1}^{N-1} \| [z_h]_n \|_{\Om{0}} + \| z_{h,N}^- \|_{\Om{0}} \bigg) \\
\le & \; C_N F_0(u) \|\theta_N^-\|_{\Om{0}}
\end{split}
\label{errrep_q0}
\end{equation} 
and for $q=1$
\begin{equation}
\begin{split}
\|\theta_N^-\|_{\Om{0}}^2 \le & \; \max_{1 \le n \le N} \bigg\{\|u - R_{n}u \|_{\Om{0}, I_n} + k_n \|\lap_{n}\{R_{n}u - \tilde{u}\}\|_{\Om{0}, I_n} \bigg\} \\
& \times \bigg( 2 \sum_{n=1}^N \int_{I_n}  \| \dot z_h \|_{\Om{0}} \ud t + \sum_{n=1}^{N-1} \| [z_h]_n \|_{\Om{0}} + \| z_{h,N}^- \|_{\Om{0}} \bigg) \\
\le & \; C_N F_1(u) \|\theta_N^-\|_{\Om{0}}
\end{split}
\label{errrep_q1}
\end{equation} 
where $F_q(u)$ is the factor with the $\max$-function. To obtain the last inequalities, we have used the stability estimate  (\ref{stabestmainzh}) with $z_{h,N}^+ = \theta_N^-$. Thus, for $q = 0, 1$
\begin{equation}
\|\theta_N^-\|_{\Om{0}} \le C_N F_q(u)
\label{theta}
\end{equation}
We note that we may write the argument to the $\max$-function in $F_q(u)$ as
%
%\begin{equation}
%\begin{split}
%F_q(u) = & \max_{1 \le n \le N} \bigg\{\|u - R_{n}u \|_{\Om{0}, I_n} + (1 - q)\| R_{n}u - \tilde{u} \|_{\Om{0}, I_n} \\
%& + q k_n\|\lap_{n}\{R_{n}u - \tilde{u}\}\|_{\Om{0}, I_n} \bigg\}
%\end{split} 
%\label{Fqinit}
%\end{equation}
\begin{equation}
\|u - R_{n}u \|_{\Om{0}, I_n} + (1 - q)\| R_{n}u - \tilde{u} \|_{\Om{0}, I_n} + q k_n\|\lap_{n}\{R_{n}u - \tilde{u}\}\|_{\Om{0}, I_n}
\label{Fqinit}
\end{equation}
We treat the terms separately, starting with the first for which we use Lemma~\ref{lem:ritzop_error}:
\begin{equation}
\|u - R_{n}u \|_{\Om{0}, I_n} \lesssim h^{p+1} \| D_x^{p+1} u \|_{\Om{0}, I_n}
\label{FqinitI}
\end{equation}
The second term in (\ref{Fqinit}) is
\begin{equation}
\begin{split}
\|R_{n}u - \tilde{u} \|_{\Om{0}, I_n} & \le \|u - R_{n}u \|_{\Om{0}, I_n} + \|\tilde{I}^n(u - R_{n}u)\|_{\Om{0}, I_n} + \|u - \tilde{I}^nu\|_{\Om{0}, I_n} \\
& \lesssim h^{p+1} \| D_x^{p+1} u \|_{\Om{0}, I_n} + k_n^{q+1}\| \dot{u}^{(q+1)} \|_{\Om{0}, I_n}
\end{split}
\label{FqinitII}
\end{equation}
For the three terms in the first row of \eqref{FqinitII}: we have used (\ref{FqinitI}) on the first term; on the second term, we have first used the boundedness of $\tilde{I}^n$ from Lemma~\ref{lemInv}, and then applied (\ref{FqinitI}); on the last term, we have used (\ref{estvInv}) from Lemma~\ref{lemInv}. We move on to the third term in (\ref{Fqinit}). Note that this term is only present for $q = 1$. To treat it we use that: For $\psi \in H^2(\Om{0})$ and $v \in V_{h, n}$, we have from \eqref{dlapdef}, \eqref{Rtdef}, and Corollary~\ref{cor:partintbroksob_A} that
\begin{equation}
(- \lap_{n}R_{n}\psi, v)_{\Om{0}} = A_n(R_{n}\psi, v) = A_n(\psi, v) = (-\lap \psi, v)_{\Om{0}}
\label{lapInRtwest_aux0}
\end{equation}
Taking $v = -\lap_{n}R_{n}\psi$ in \eqref{lapInRtwest_aux0} gives $\| \lap_{n}R_{n}\psi \|_{\Om{0}} \leq \| \lap \psi \|_{\Om{0}}$ which we use after applying Lemma~\ref{lemInv} to the third term in (\ref{Fqinit}). Thus
\begin{equation}
\|\lap_{n}\{R_{n}u - \tilde{u}\} \|_{\Om{0}, I_n} \lesssim k_n^2 \|\lap_{n}R_n \dot{u}^{(2)} \|_{\Om{0}, I_n} \leq k_n^2 \|\lap \dot{u}^{(2)} \|_{\Om{0}, I_n} = k_n^2 \| \dot{u}^{(3)} \|_{\Om{0}, I_n}
\label{FqinitIII}
\end{equation}
Collecting the results, we get for $q = 0, 1$
\begin{equation}
F_q(u) \lesssim \max_{1 \le n \le N} \bigg\{ h^{p+1} \| D_x^{p+1} u \|_{\Om{0}, I_n} + k_n^{2q+1}\| \dot{u}^{(2q+1)} \|_{\Om{0}, I_n} \bigg\}
\label{Fqfin}
\end{equation} 
We note from \eqref{rhonm} that $\| \rho_N^- \|_{\Om{0}} \leq F_q(u)$. Combining this with \eqref{theta} and using \eqref{Fqfin} concludes the proof of Theorem~\ref{aprithm}.
\end{proof}

\section{Numerical results} \label{secnumres}

%\noindent Here we present numerical results for the following model problem in one spatial dimension: 
%
%\begin{equation} 
%\left\{
%\begin{split}
%\dot{u} - u_{xx}  & =  -(\frac{1}{2}\sin^2(\pi x) + 2\pi^2\cos(2\pi x))\euler^{-t/2} && \text{in} \ (0,1) \times (0, T], \\
%u & = 0 && \text{on} \ \ \{0,1\} \times (0, T], \\
%u & = \sin^2(\pi x) && \text{in} \ (0,1) \times \{0\}. 
%\end{split}
%\right. \label{numresheateq}
%\end{equation}
%
%The exact solution of (\ref{numresheateq}) is $u = \sin^2(\pi x)\euler^{-t/2}$.
Here we present numerical results for a problem in one spatial dimension on the unit interval with exact solution $u(x, t) = \sin^2(\pi x)\euler^{-t/2}$. We compute $u_h$ for $p=1$ and $q = 0,1$. The right-hand side integrals have been approximated locally by quadrature over the space-time prisms: first quadrature in time, then quadrature in space. In space, \emph{three}-point Gauss-Legendre quadrature has been used, thus resulting in a quadrature error $= O(h^6)$. For dG(0) in time, the midpoint rule has been used, thus resulting in a quadrature error $= O(k^2)$. For dG(1) in time, \emph{three}-point Lobatto quadrature has been used, thus resulting in a quadrature error $= O(k^4)$. The stabilization parameter $\gamma = 10$. \\

For the error convergence study, both $\mathcal{T}_0$ and $\mathcal{T}_G$ are uniform meshes, with mesh sizes $h_0$ and $h_G$, respectively. The shape of $\mathcal{T}_G$ is kept fixed to conveniently avoid introducing a change in $h_G$. The temporal discretization is also uniform with time step $k$ for each instance. The final time is set to $T = 1$, the length of $\mathcal{T}_G$ is 0.25, and the initial position of $\mathcal{T}_G$ is the spatial interval [0.125, 0.125 + 0.25]. The $error$ is $\| e(T) \|_{L^2(\Om{0})} = \| u(T) - u_{h, N}^-\|_{\Om{0}}$. The integral in the $L^2$-norm has been approximated by composite \emph{three}-point Gauss-Legendre quadrature, thus resulting in a quadrature error $= O(h^3)$. In the $k$-convergence study, the mesh sizes have been fixed at $h = 10^{-3}$ and $h = 5 \cdot 10^{-5}$ for dG(0) and dG(1), respectively. Analogously, in the $h$-convergence study, the time step has been fixed at $k = 10^{-4}$ and $k = 10^{-3}$ for dG(0) and dG(1), respectively. Figure~\ref{fig_dG0_ECC_mu0p6} and \ref{fig_dG1_ECC_mu0p6} display error convergence plots for dG(0) and dG(1) in time with $\mu=0.6$. The left plots show the $error$ versus $k$ and the right plots versus $h = h_0 \geq h_G$. Besides the computed $error$, each plot contains a line segment that has been computed with the linear least squares method to fit the error data. This line segment is referred to as the LLS of the $error$. Reference slopes are also included. In Table~\ref{tabnumord} we summarize the slope of the LLS of the $error$ for different values of $\mu$. 
%
%\begin{figure}[h]
%\centering
%\includegraphics[width=0.4\textwidth]{figures/plots/dG0/error/error_k_mu0p0_hfix1em3_15points_LLS1p0064_points1to15used-eps-converted-to.pdf}
%\includegraphics[width=0.4\textwidth]{figures/plots/dG0/error/error_h_mu0p0_kfix1em4_15points_LLS2p0559_points1to11used-eps-converted-to.pdf}
%\caption{Error convergence for dG(0) with $\mu = 0$. %The slopes of the LLS are 1.0064 (left) and 2.0559 (right).
%\label{fig_dG0_ECC_mu0p0}}
%\end{figure}
%
\begin{figure}[h]
\centering
\includegraphics[width=0.4\textwidth]{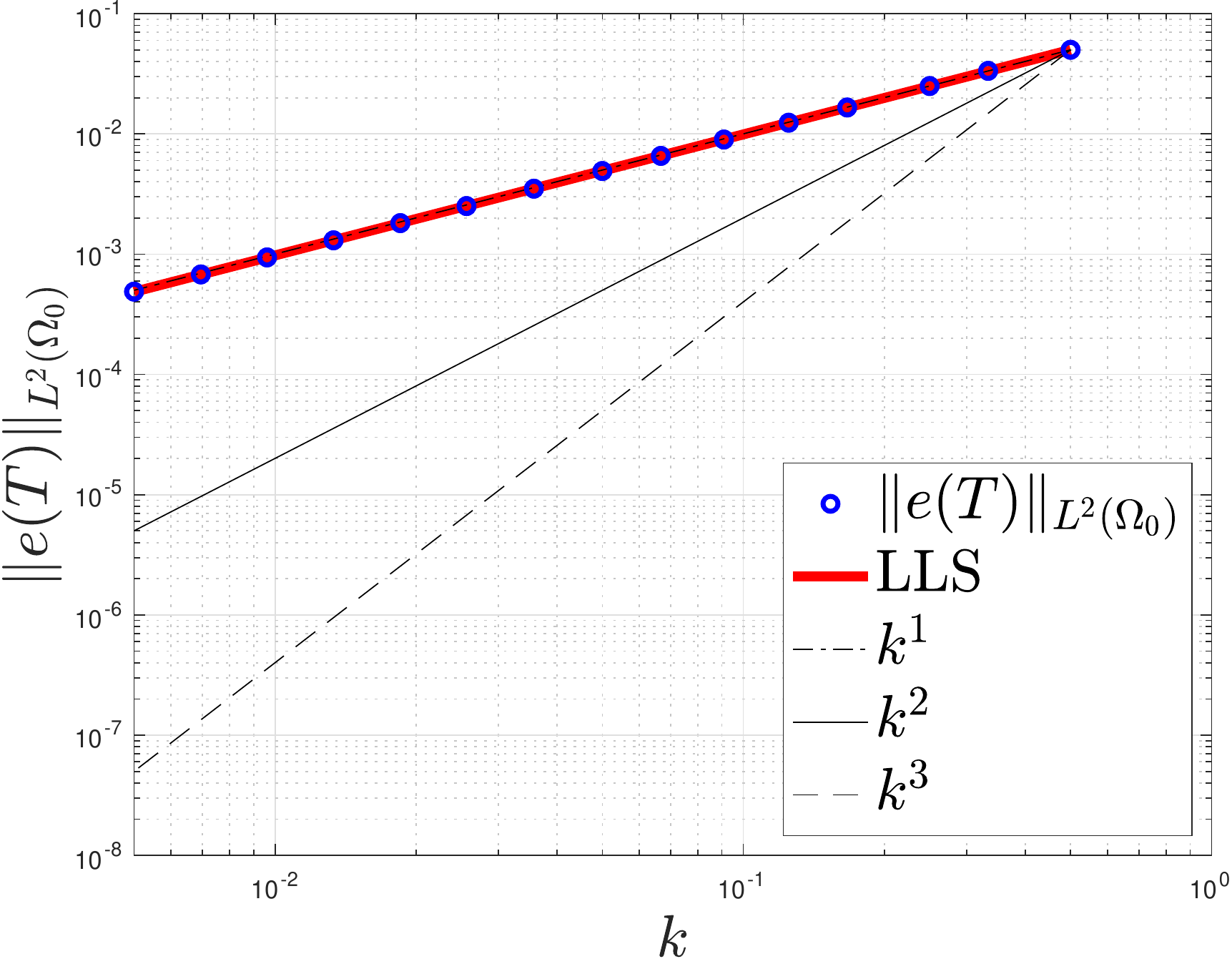}
\includegraphics[width=0.4\textwidth]{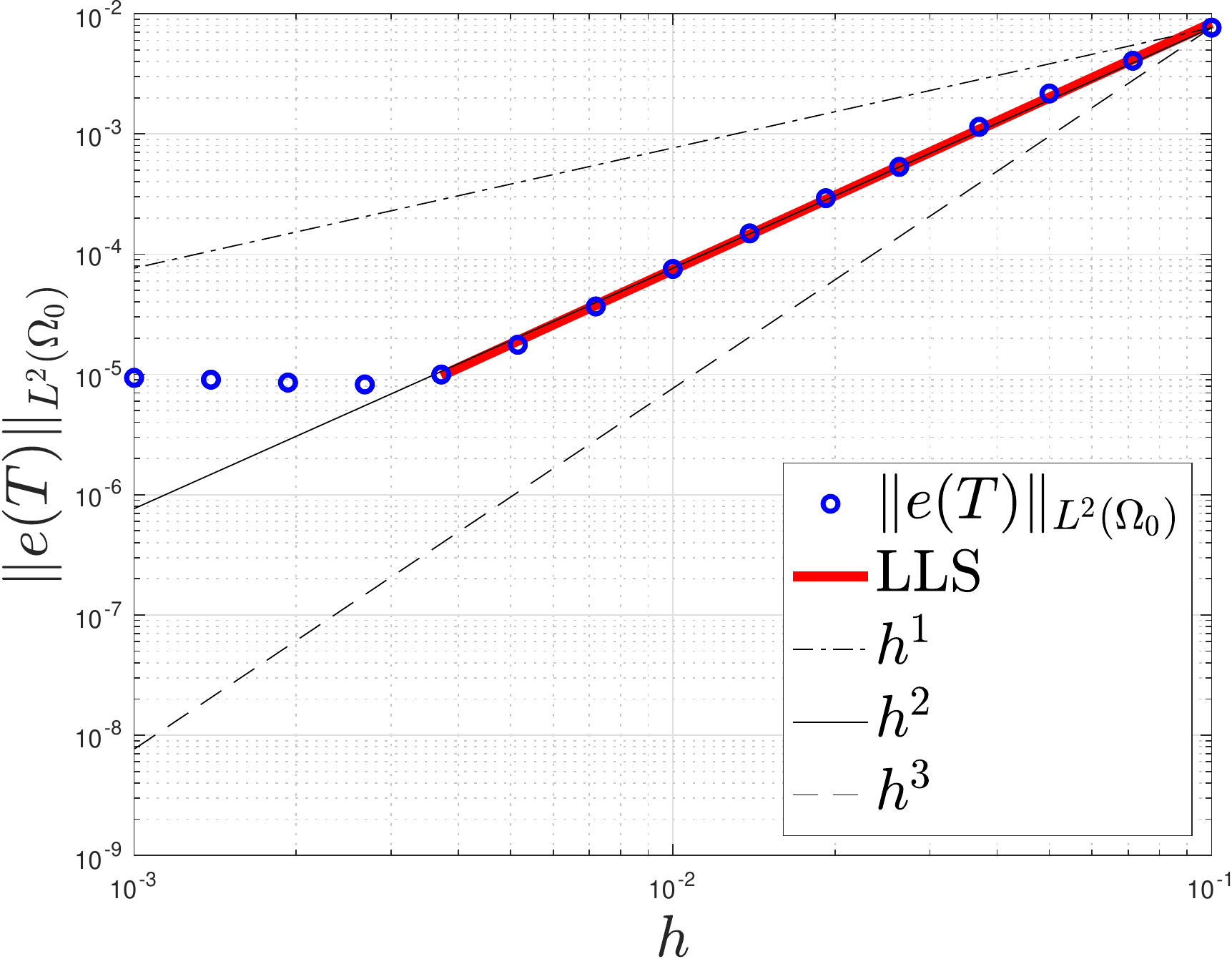}
\caption{Error convergence for dG(0) with $\mu = 0.6$. %The slopes of the LLS are 1.0064 (left) and 2.0501 (right).
\label{fig_dG0_ECC_mu0p6}}
\end{figure}
%
%\begin{figure}[h]
%\centering
%\includegraphics[width=0.4\textwidth]{figures/plots/dG1/error/error_k_mu0p0_hfix5em5_15points_LLS2p7890_points9to12used-eps-converted-to.pdf}
%\includegraphics[width=0.4\textwidth]{figures/plots/dG1/error/error_h_mu0p0_kfix1em3_15points_LLS2p0122_points1to15used-eps-converted-to.pdf}
%\caption{Error convergence for dG(1) with $\mu = 0$. %The slopes of the LLS are 2.7890 (left) and 2.0122 (right). 
%\label{fig_dG1_ECC_mu0p0}}
%\end{figure}
%
\begin{figure}[h]
\centering
\includegraphics[width=0.4\textwidth]{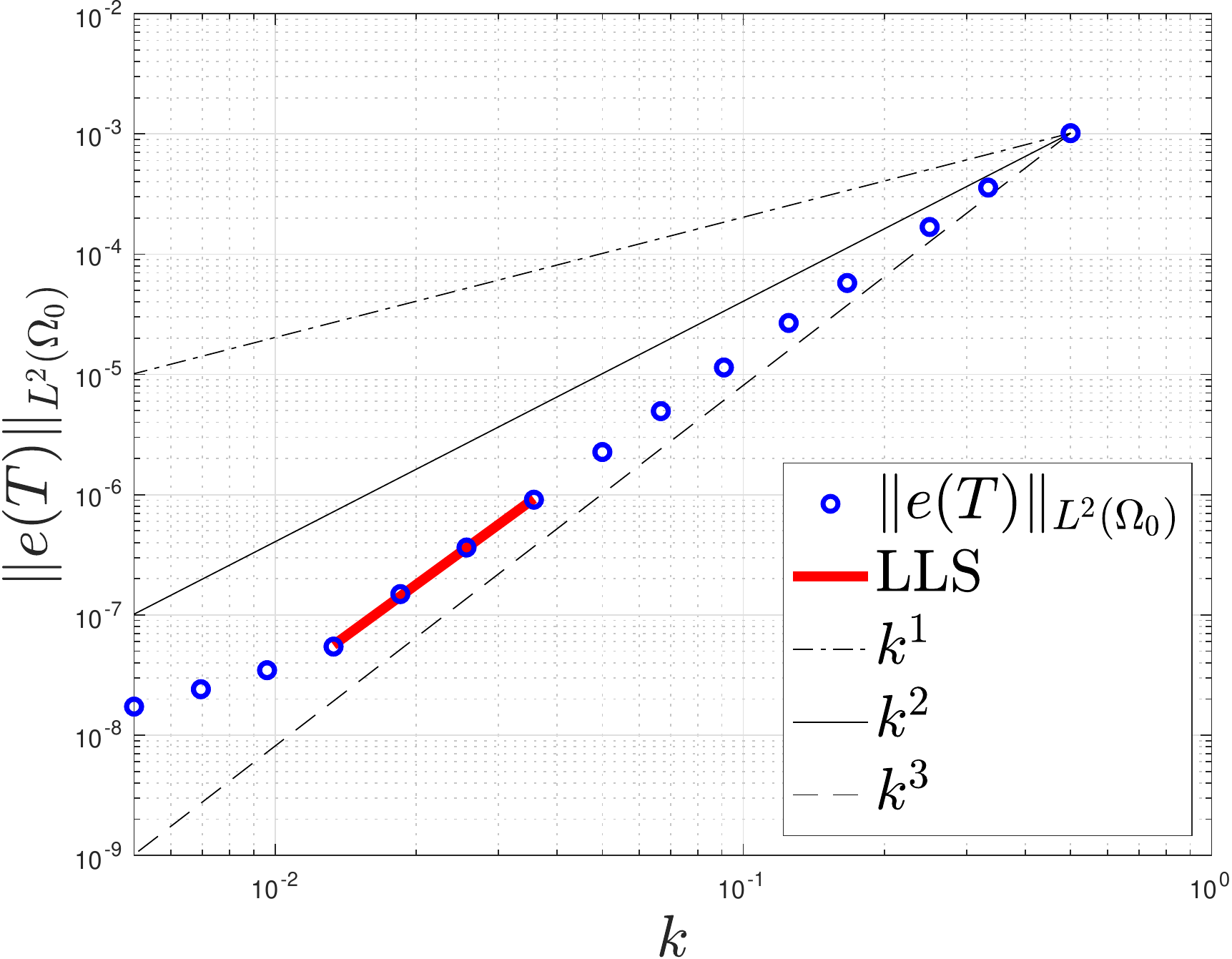}
\includegraphics[width=0.4\textwidth]{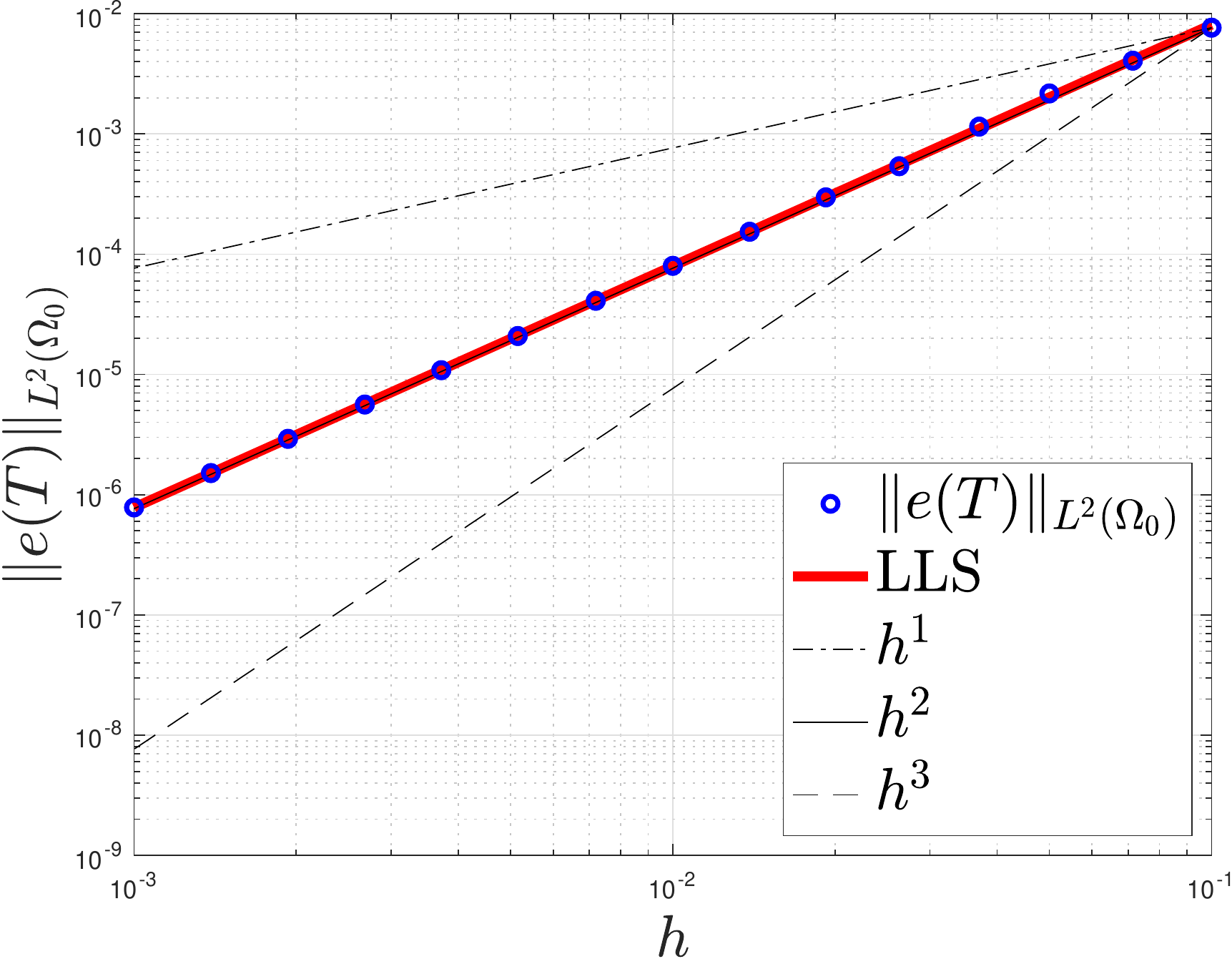}
\caption{Error convergence for dG(1) with $\mu = 0.6$. %The slopes of the LLS are 2.8437 (left) and 2.0082 (right).
\label{fig_dG1_ECC_mu0p6}}
\end{figure}
\begin{table}[h]
\centering
\begin{tabular}{ l | c | c | c | c | } 
  %& \multicolumn{4}{|c|}{Slope of the LLS of the $error$} \\  
  & \multicolumn{2}{|c|}{dG(0) in time} & \multicolumn{2}{|c|}{dG(1) in time} \\  
  \hline                    
  $\mu$ & versus $k$ (points) & versus $h$ (points) & versus $k$ (points) & versus $h$ (points) \\
  \hline
  0 & 1.0064 (1--15) & 2.0559 (1--11) & 2.7890 (9--12) & 2.0122 (1--15) \\
  \hline
  0.1 & 1.0064 (1--15) & 2.0486 (1--11) & 2.9142 (9--12) & 2.0058 (1--15) \\ 
  \hline
  0.2 & 1.0064 (1--15) & 2.0421 (1--11) & 2.8493 (9--12) & 2.0024 (1--15) \\ 
  \hline 
  0.4 & 1.0064 (1--15) & 2.0422 (1--11) & 2.6994 (9--12) & 2.0024 (1--15) \\ 
  \hline 
  0.6 & 1.0064 (1--15) & 2.0501 (1--11) & 2.8437 (9--12) & 2.0082 (1--15) \\
  \hline 
\end{tabular}
\caption{The slope of the LLS of the $error$ versus $k$ and $h$ for different values of $\mu$.
\label{tabnumord}}
\end{table}

%The solution is presented for two different pairs of equidistant space-time discretizations, where $G$ is immersed in $\Omega_0$ for all $t \in [0, 3]$, and the length of $G$ is 0.25. First, we consider \emph{the coarse case}: 22 nodes for $\mathcal{T}_0$, 7 nodes for $\mathcal{T}_G$, and 10 time steps on the interval $(0, 3]$. Second, we consider \emph{the fine case}: 44 nodes for $\mathcal{T}_0$, 14 nodes for $\mathcal{T}_G$, and 30 time steps on the interval $(0, 3]$. We present the solution for these two cases for three different velocities ($\mu$) of the overlapping mesh $\mathcal{T}_G$: $\mu = 0$, $\mu = 0.1$, and $\mu = \frac{1}{2}\sin(\frac{2 \pi t}{3})$.

The numerical solutions presented in Figure~\ref{fig_numsol} have been computed for an equidistant space-time discretization: 22 nodes for $\mathcal{T}_0$, 7 nodes for $\mathcal{T}_G$ for all times, and 10 time steps on the interval $(0, 3]$. The length of $\mathcal{T}_G$ has again been kept fixed at 0.25 and the velocity field $\mu$ has for simplicity been slabwise constant at $\mu|_{I_n} = \frac{1}{2}\sin(\frac{2 \pi t_n}{3})$.
\begin{figure}[h]
\centering
\includegraphics[width=0.32\textwidth]{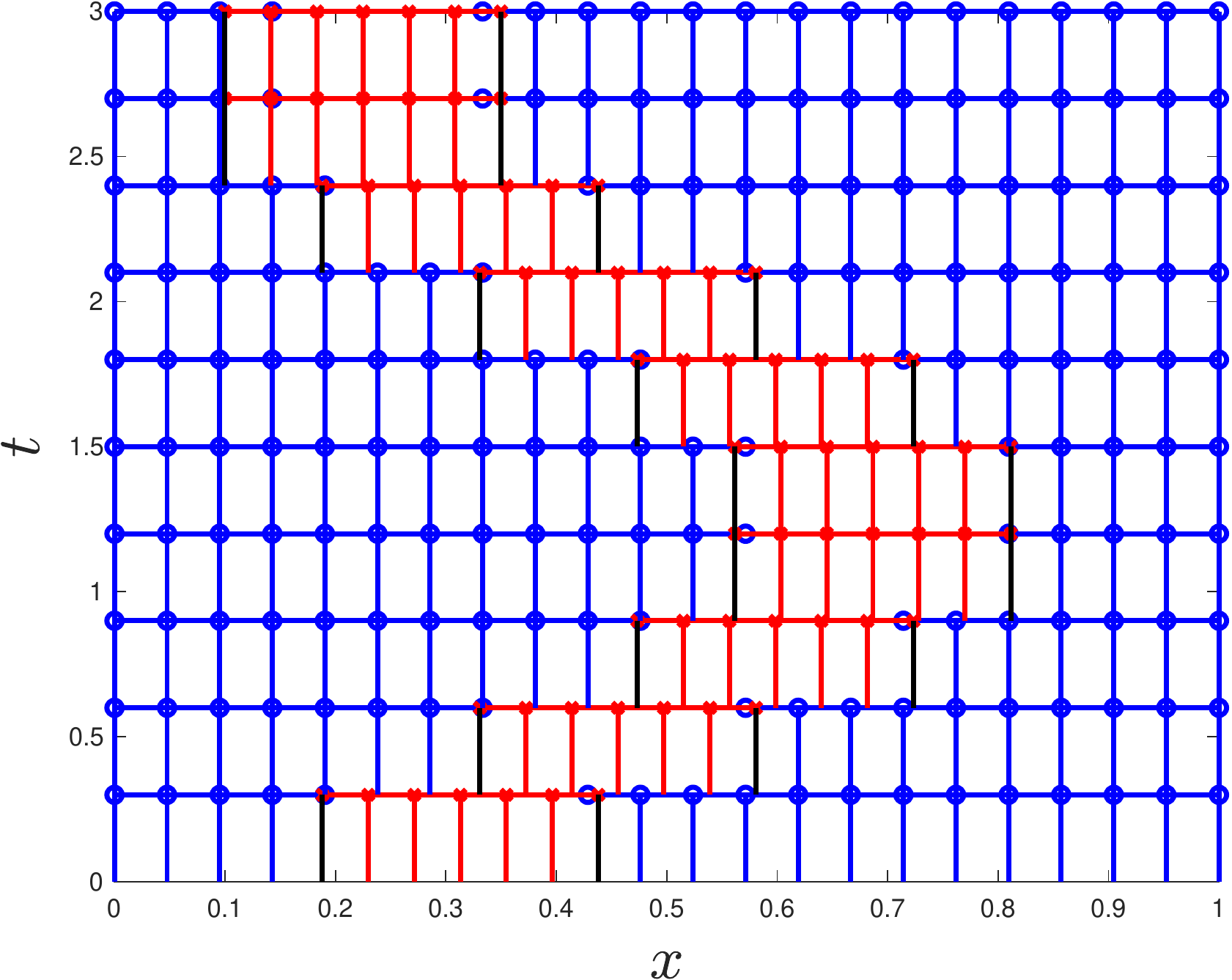}
\includegraphics[width=0.32\textwidth]{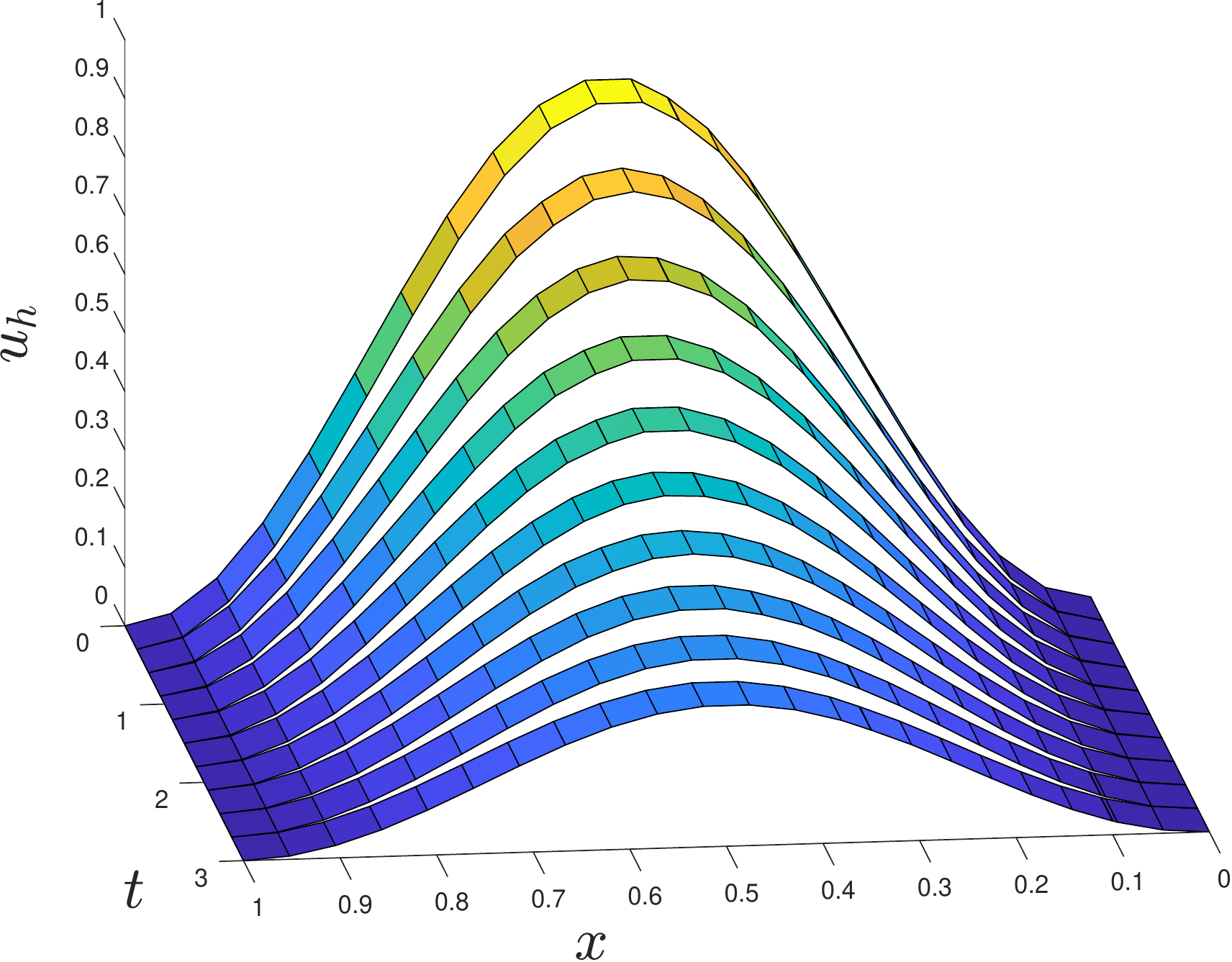}
\includegraphics[width=0.32\textwidth]{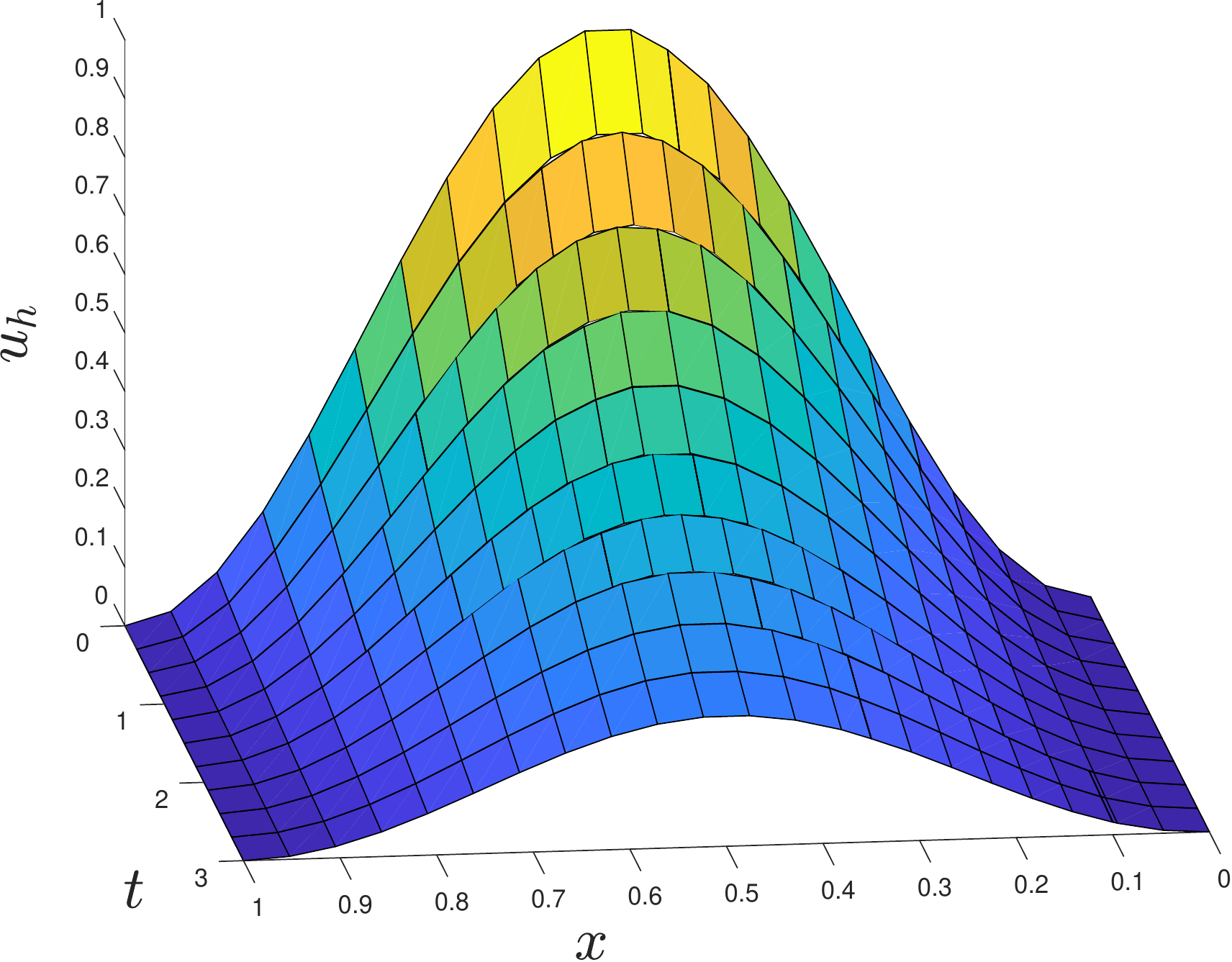}
\caption{Space-time discretization (left) with resulting dG(0)cG(1)-solution (middle) and dG(1)cG(1)-solution (right). \label{fig_numsol}}
\end{figure}

\section{Conclusions}

We have presented a cut finite element method for a parabolic model problem on an overlapping mesh situation: one stationary background mesh and one \emph{discontinuously} evolving, slabwise stationary overlapping mesh. We have applied the analysis framework presented in \cite{Eriksson1991, Eriksson1995} to the method with natural modifications to account for the CutFEM setting. The greatest difference and novelty in the presented analysis is the shift operator. The main results of the analysis are basic and strong stability estimates and an optimal order a priori error estimate. We have also presented numerical results for a parabolic problem in one spatial dimension that verify the analytic error convergence orders.

\appendix
\section{Analytic tools}

\begin{lemma}[A jump identity] \label{jmplem}
Let $\omega_+, \omega_- \in \mathbb{R}$ and $\omega_+ + \omega_- = 1$, let $[A] := A_+ - A_-$, and $\langle A \rangle := \omega_+A_+  + \omega_-A_-$. We then have
\begin{equation}
[AB] =  [A]\langle B \rangle + \langle A \rangle [B] + (\omega_- - \omega_+)[A][B]
\label{jmplemres}
\end{equation}
\begin{proof}
Using the definitions and evaluating both sides shows the identity.
\end{proof}
\end{lemma}

\begin{lemma}[Partial integration in broken Sobolev spaces]\label{lem:partintbroksob}
For $d = 1, 2$, or $3$, let $\Omega \subset \mathbb{R}^d$ be a bounded domain and let $\Gamma \subset \Omega$ be a continuous manifold of codimension 1 that partitions $\Omega$ into the subdomains $\Om{1}, \cdots, \Om{N}$. For $\psi \in H^2(\Omega)$ and $v \in H^1_0(\Om{1}, \cdots, \Om{N})$, we have that
\begin{equation}
(-\lap \psi, v)_{\Omega} = \sum_{i=1}^N(\nab \psi, \nab v)_{\Om{i}} - (\langle \partial_n \psi \rangle, [v])_{\Gamma}
\label{lemres:partintbroksob}
\end{equation}
\begin{proof}
Using the partition of $\Omega$ and Green's first identity on the left-hand side gives interior terms and boundary terms. The interior terms are as we want them. For the boundary terms, only the $\Gamma$-terms remain since $v|_{\partial \Omega} = 0$. Combining terms with common boundary, using Lemma~\ref{jmplem} and the regularity of $\psi$ shows \eqref{lemres:partintbroksob}.
\end{proof}
\end{lemma}
\noindent Consider the domain partition and its corresponding broken Sobolev space presented in the premise of Lemma~\ref{lem:partintbroksob}. We define the symmetric bilinear form $\mathcal{A}$ that generalizes the appearence of $A_{h, t}$ defined by (\ref{Ahdef}) to this setting by
\begin{equation}
\begin{split}
\mathcal{A}(w, v) := & \sum_{i=1}^N (\nab w, \nab v)_{\Om{i}} - (\langle \partial_{n} w \rangle, [v])_{\Gamma} - (\langle \partial_{n} v \rangle, [w])_{\Gamma} \\
& + (\gamma h_K^{-1} [w],[v])_{\Gamma} + ([\nab w],[\nab v])_{\Om{O}}
\end{split} 
\label{genAdef}
\end{equation}
where we just let $h_K^{-1}$ be some spatially dependent function of sufficient regularity and $\Om{O}$ be some union of subsets of subdomains. The specifics of $h_K^{-1}$ and $\Om{O}$ are of course taken to be the natural ones when restricting $\mathcal{A}$ to $A_{h, t}$. For $\psi \in H^2(\Omega)$, we have from regularity that $[\psi]|_\Gamma = 0$ in $L^2(\Gamma)$ and that $[\nab \psi]|_{\Om{O}} = 0$ since $(\nab \psi)_+ = (\nab \psi)_-$ on $\Om{O}$ for a non-discrete function such as $\psi$. Using this, we combine $\mathcal{A}$ with Lemma~\ref{lem:partintbroksob} to get the following corollary:
\begin{corollary}[Partial integration in broken Sobolev spaces with bilinear forms $\mathcal{A}$] \label{cor:partintbroksob_A}
For $d = 1, 2$, or $3$, let $\Omega \subset \mathbb{R}^d$ be a bounded domain and let $\Gamma \subset \Omega$ be a continuous manifold of codimension 1 that partitions $\Omega$ into the subdomains $\Om{1}, \cdots, \Om{N}$. For this setting, let the symmetric bilinear form $\mathcal{A}$ be defined by \eqref{genAdef}. For $\psi \in H^2(\Omega)$ and $v \in H^1_0(\Om{1}, \cdots, \Om{N})$, we have that
\begin{equation}
(-\lap \psi, v)_{\Omega} = \mathcal{A}(\psi, v)
\label{corres:partintbroksob_A}
\end{equation}
\end{corollary}

\begin{lemma}[A scaled trace inequality for domain-partitioning manifolds of codimension 1] \label{lem:scatraineqdomcut}
For $d = 1, 2$, or $3$, let $\Omega \subset \mathbb{R}^d$ be a bounded domain with diameter $L$, i.e., $L = \text{diam}(\Omega) = \sup_{x, y \in \Omega} |x - y|$. Let $\Gamma \subset \Omega$ be a continuous manifold of codimension 1 that partitions $\Omega$ into $N$ subdomains. Then
\begin{equation}
\| v \|_{\Gamma}^2 \lesssim L^{-1}\| v \|_{\Omega}^2 + L \| \nab v \|_{\Omega}^2 \quad \forall v \in H^1(\Omega)
\label{lemres:scatraineqdomcut}
\end{equation}
\begin{proof}
If (\ref{lemres:scatraineqdomcut}) holds for the case $N = 2$, then that result may be applied repeatedly to show (\ref{lemres:scatraineqdomcut}) for $N > 2$. We thus assume that $\Gamma$ partitions $\Omega$ into two subdomains denoted $\Om{1}$ and $\Om{2}$ with diameters $L_1$ and $L_2$, respectively. From the regularity assumptions on $v$, we have for $i = 1, 2$, that $v \in H^1(\Om{i})$ and thus 
\begin{equation}
\| v \|_{\Gamma}^2 \leq \| v \|_{\partial \Om{i}}^2 \lesssim L_i^{-1}\| v \|_{\Om{i}}^2 + L_i \| \nab v \|_{\Om{i}}^2
\label{scatraineqdomcut_standard}
\end{equation}
where we have used a standard scaled trace inequality. Using the triangle type inequality $L \leq L_1 + L_2$ and \eqref{scatraineqdomcut_standard}, the left-hand side of \eqref{lemres:scatraineqdomcut} is
\begin{equation}
\| v \|_{\Gamma}^2 \leq \sum_{i=1}^2 \frac{L_i}{L}\| v \|_{\Gamma}^2 \lesssim \sum_{i=1}^2 \bigg(L^{-1}\| v \|_{\Om{i}}^2 + L \| \nab v \|_{\Om{i}}^2 \bigg) \lesssim L^{-1}\| v \|_{\Omega}^2 + L\| \nab v \|_{\Omega}^2
\label{scatraineqdomcut_0}
\end{equation}
which shows \eqref{lemres:scatraineqdomcut}.
\end{proof}
\end{lemma}
\noindent Let $\Gamma_K = \Gamma_K(t) = K \cap \Gamma(t)$. For $t \in [0, T]$, $j \in \{0, G\}$, a simplex $K \in \mathcal{T}_{j,\Gamma(t)} = \{K \in \mathcal{T}_j : K \cap \Gamma(t) \neq \emptyset\}$, and $v \in H^1(K)$, we have from Lemma~\ref{lem:scatraineqdomcut} that
\begin{equation}
\| v \|_{\Gamma_K}^2 \lesssim h_K^{-1}\| v \|_{K}^2 + h_K \| \nab v \|_{K}^2
\label{scatraineqGamK_warmup}
\end{equation}
where $h_K$ is the diameter of $K$. For $v \in \mathcal{P}(K)$, we have the standard inverse estimate
\begin{equation}
\| D_x^{k} v \|^2_K \lesssim h_K^{-2} \| D_x^{k-1} v \|^2_K \quad \text{ for } k \geq 1
\label{investpolK_standard}
\end{equation}
Using \eqref{investpolK_standard} in \eqref{scatraineqGamK_warmup}, we get the following corollary:
\begin{corollary}[A discrete spatial local inverse inequality for $\Gamma_K(t)$] \label{cor:scatraineqGamK}
For $t \in [0, T]$, $j \in \{0, G\}$, $K \in \mathcal{T}_{j,\Gamma(t)}$ with diameter $h_K$, let $\Gamma_K(t) = K \cap \Gamma(t)$. Then, for $k \geq 0$, we have that
\begin{equation}
\| D_x^k v \|_{\Gamma_K(t)}^2 \lesssim h_K^{-1} \| D_x^k v  \|_{K}^2 \quad \forall v \in V_h(t)
\label{corres:scatraineqGamK}
\end{equation}
\end{corollary}

\begin{lemma}[A discrete spatial inverse inequality for $\Gamma(t)$] \label{invineqgamlem}
Let the mesh-dependent norm $\| \cdot \|_{-1/2,h,\Gamma(t)}$ be defined by $(\ref{def:HHnorm})$. Then, for $t \in [0, T]$, we have that 
\begin{equation}
\| \langle \partial_{n} v \rangle \|_{-1/2,h,\Gamma(t)}^2 \lesssim \sum_{i=1}^2 \|\nab v\|_{\Omt{i}}^2 + \| [\nab v] \|_{\Omt{O}}^2 \quad \forall v \in V_h(t)
\label{invineqgamlemres}
\end{equation} 
\begin{proof}
To lighten the notation we omit the time dependence, which has no importance here anyways. We follow the proof of the corresponding inequality in \cite{Hansbo:2002aa} with some modifications. We use index $j \in \{0, G \}$, such that, if $j = 0$, then $i = 1$ and if $j = G$, then $i = 2$, and let $\Gamma_{K_j} = K_j \cap \Gamma$ and $\mathcal{T}_{j,\Gamma} = \{K_j \in \mathcal{T}_j : K_j \cap \Gamma \neq \emptyset\}$. Note that for $i=1, 2$,
\begin{equation}
\sum_{K_0 \in \mathcal{T}_{0,\Gamma}} h_{K_0} \| v_i \|_{\Gamma_{K_0}}^2 \lesssim \sum_{K_G \in \mathcal{T}_{G,\Gamma}} h_{K_G} \| v_i \|_{\Gamma_{K_G}}^2
\label{gamk0ineqgamkg}
\end{equation}
which follows from $\cup_{K_0 \in \mathcal{T}_{0,\Gamma}} \Gamma_{K_0} = \Gamma = \cup_{K_G \in \mathcal{T}_{G,\Gamma}} \Gamma_{K_G}$ and the inter-quasi-uniformity of the meshes. Since $\partial_{n} v = n \cdot \nab v$ and $|\omega_i| |n| \leq 1$, we have $\| \omega_i (\partial_{n} v)_i \|_{\Gamma_{K_j}}^2 \leq \|(\nab v)_i \|_{\Gamma_{K_j}}^2$. Using this after \eqref{gamk0ineqgamkg}, and followed by  Corollary~\ref{cor:scatraineqGamK}, the left-hand side of (\ref{invineqgamlemres}) is 
\begin{equation}
\begin{split}
\| \langle \partial_{n} v \rangle \|_{-1/2,h,\Gamma}^2 & \lesssim \sum_{i=1}^2 \sum_{K_j \in \mathcal{T}_{j, \Gamma}} h_{K_j} \|(\nab v)_i \|_{\Gamma_{K_j}}^2 \lesssim \sum_{i=1}^2 \sum_{K_j \in \mathcal{T}_{j, \Gamma}} \|(\nab v)_i\|_{K_j}^2 \\
& = \sum_{K_0 \in \mathcal{T}_{0,\Gamma}} \bigg( \|\nab v\|_{K_0 \cap \Om{1}}^2 +  \|(\nab v)_1\|_{K_0 \cap \Om{2}}^2 \bigg) + \sum_{K_G \in \mathcal{T}_{G,\Gamma}} \|\nab v\|_{K_G}^2 %\\
%& \lesssim \sum_{i=1}^2\|\nab v\|_{\Om{i}}^2 + \| [\nab v] \|_{\Om{O}}^2
\end{split} 
\label{invineqgamlemfin}
\end{equation}
% 
%which is the desired estimate.
The resulting terms may be estimated by the right-hand side of \eqref{invineqgamlemres}.
\end{proof}
\end{lemma}

\section{Interpolation} \label{sec:interpolation}

%For a function space $V$ with functions that are zero on $\partial \Om{0}$, we denote by $\widehat{V}$ the larger space obtained by removing the constraint $v|_{\partial \Om{0}} = 0$. For the definition of the spatial interpolation operator, we recall the semi-discrete spaces $V_{h,0}$ and $V_{h,G}$, defined by (\ref{fesvht0}) and  (\ref{fesvhtG}), respectively. We define the spatial interpolation operators $\pi_{h,0} : L^1(\Om{0}) \to \widehat{V}_{h,0}$ and $\pi_{h,G} : L^1(G) \to \widehat{V}_{h,G}$ to be the Scott--Zhang interpolation operators for the spaces $\widehat{V}_{h,0}$ and $\widehat{V}_{h,G}$, respectively. We point out that $\pi_{h,G} = \pi_{h,G}(t)$, i.e., it is time-dependent, since $G$ is allowed to move around, but we omit the $t$ to lighten the notation. For $t \in [0, T]$, we define the spatial interpolation operator $I_{h,t} : L^1(\Om{0}) \to \widehat{V}_h(t)$ by
For the definition of the spatial interpolation operator, we recall the discrete spaces $V_{h,0}$ and $V_{h,G}$. We define the spatial interpolation operators $\pi_{h,0} : L^1(\Om{0}) \to V_{h,0}$ and $\pi_{h,G} : L^1(G) \to V_{h,G}$ to be the Scott--Zhang interpolation operators for the spaces $V_{h,0}$ and $V_{h,G}$, respectively, where the defining integrals are taken over entire simplices. We point out that the evolution of $G$ makes $\pi_{h,G}$ time-dependent, but since that does not matter here we omit it to lighten the notation. For $t \in [0, T]$, we define the spatial interpolation operator $I_{h,t} : L^1(\Om{0}) \to V_h(t)$ by
\begin{equation}
I_{h,t} v|_{\Omt{1}} := \pi_{h,0}v|_{\Omt{1}}, \quad \quad I_{h,t} v|_{\Omt{2}} := \pi_{h,G}v|_{\Omt{2}}
\label{def:interph}
\end{equation} 
The operator $I_{h,t}$ is used in the proofs of Lemma~\ref{lem:ritzop_error} and Lemma~\ref{lem:shiftop_error}, where energy estimates of its interpolation error are used. We present these estimates in the following two lemmas:
\begin{lemma}[An interpolation error estimate in $\norma{\cdot}$]
\label{lem:interphest_energy}
Let $\norma{\cdot}$ and $I_{h,t}$ be defined by \eqref{def:anorm} and \eqref{def:interph}, respectively. Then 
\begin{equation}
\norma{v - I_{h,t} v} \lesssim h^{p}\| D_x^{p+1}v\|_{\Om{0}} \quad \forall v \in H^{p+1}(\Om{0})
\label{lemres:interphest_energy}   
\end{equation} 
\begin{proof} To lighten the notation we omit the time dependence, which has no importance here anyways. Letting $w = v - I_{h,t} v$, and using the definition of $\norma{\cdot}$, the square of the left-hand side of (\ref{lemres:interphest_energy}) is 
\begin{equation}
\norma{w}^2 =  \sum_{i = 1}^2 \| \nab w \|_{\Om{i}}^2 + \|\langle \partial_{n} w \rangle \|_{-1/2,h,\Gamma}^2 + \|[w] \|_{1/2,h,\Gamma}^2 + \|[\nab w]\|_{\Om{O}}^2
\label{interphest_energy0}   
\end{equation} 
Letting $w_j = v - \pi_{h, j}v$, we treat each term in \eqref{interphest_energy0} separately, starting with the first:
\begin{equation}
\| \nab w \|_{\Om{i}}^2 \leq \sum_{K \in \mathcal{T}_{j, \Om{i}}} \| \nab w_j \|_{K}^2
\label{interphest_energy0_1}   
\end{equation} 
Following the proof of Lemma~\ref{invineqgamlem} and using \eqref{scatraineqGamK_warmup}, the second term is
\begin{equation}
\|\langle \partial_{n} w \rangle \|_{-1/2,h,\Gamma}^2 \lesssim \sum_{\substack{i \in \{1, 2\}\\ K_j \in \mathcal{T}_{j, \Gamma}}} h_{K_j} \| (\nab w)_i \|_{\Gamma_{K_j}}^2 \lesssim \sum_{\substack{i \in \{1, 2\}\\ K \in \mathcal{T}_{j, \Om{i}}}} \bigg(\| \nab w_j \|_{K}^2 + h_{K}^2\| D_x^2 w_j \|_{K}^2 \bigg)
\label{interphest_energy0_2}
\end{equation}
The third term in (\ref{interphest_energy0}) receives the same treatment, thus 
\begin{equation}
\|[w] \|_{1/2,h,\Gamma}^2 \lesssim \sum_{\substack{i \in \{1, 2\}\\ K_j \in \mathcal{T}_{j, \Gamma}}} h_{K_j}^{-1} \| w_i \|_{\Gamma_{K_j}}^2 \lesssim \sum_{\substack{i \in \{1, 2\}\\ K \in \mathcal{T}_{j, \Om{i}}}} \bigg( h_{K}^{-2}\| w_j \|_{K}^2 + \| \nab w_j \|_{K}^2 \bigg)
\label{interphest_energy0_3}
\end{equation} 
The fourth term in (\ref{interphest_energy0}) is
\begin{equation}
\|[\nab w]\|_{\Om{O}}^2 \lesssim \sum_{i=1}^2\|(\nab w)_i \|_{\Om{O}}^2 \lesssim \sum_{i=1}^2  \sum_{K \in \mathcal{T}_{j, \Om{i}}} \| \nab w_j \|_{K}^2
\label{interphest_energy0_4}   
\end{equation} 
Summing up what we have, i.e., using (\ref{interphest_energy0_1})--(\ref{interphest_energy0_4}) in (\ref{interphest_energy0}), we get 
\begin{equation}
\norma{w}^2 \lesssim \sum_{i=1}^2  \sum_{K \in \mathcal{T}_{j, \Om{i}}} \bigg(h_{K}^{-2}\| w_j \|_{K}^2 + \| \nab w_j \|_{K}^2 + h_{K}^2\| D_x^2 w_j \|_{K}^2 \bigg)
\lesssim h^{2p} \| D_x^{p+1} v \|_{\Om{0}}^2
\label{interphest_energy_fin}   
\end{equation}
where we have used standard local interpolation error estimates for Scott--Zhang interpolation operators. Taking the square root of both sides gives (\ref{lemres:interphest_energy}).
\end{proof}
\end{lemma}

\begin{lemma}[An interpolation error estimate in $\normaspecnm{\cdot}$] \label{lem:interphest_energy_special}
For $n = 1, \dots, N$, let $\normaspecnm{\cdot}$ and $I_{h, n} = I_{h,t_n}$ be defined by \eqref{def:normaspecnm} and \eqref{def:interph}, respectively. Then
\begin{equation}
\normaspecnm{v - I_{h,n} v} \lesssim h^{p}\| D_x^{p+1}v\|_{\Om{0}} \quad \forall v \in H^{p+1}(\Om{0})
\label{lemres:interphest_energy_special}   
\end{equation} 
\begin{proof} 
Letting $w = v - I_{h,n} v$, and plugging $w$ into $\normaspecnm{\cdot}^2$, we have
\begin{equation}
\begin{split}
\normaspecnm{w}^2 & \lesssim \norman{w}^2 + \|\langle \partial_{n} w \rangle \|_{-1/2,h,\Gamma_{m} \setminus \Gamma_{n}}^2 
\end{split} 
\label{interphest_energy_special_init}   
\end{equation}  
The second term is treated by following the proof of Lemma~\ref{invineqgamlem}. We partition $\Gamma_{m} \setminus \Gamma_{n}$ into $\grave{\Gamma}_{i} := (\Gamma_{m} \setminus \Gamma_{n}) \cap \Om{{i,n}}$, use the interdependent indices $i$ and $j$, and write $\grave{\Gamma}_{i {K_j}} = K_j \cap \grave{\Gamma}_i$. Letting $w_j = v - \pi_{h, j}v$, we have
\begin{equation}
\begin{split}
\| \langle \partial_{n} w \rangle \|_{-1/2,h,\Gamma_{m} \setminus \Gamma_{n}}^2 & \lesssim \sum_{\substack{\grave{\Gamma}_{i {K_j}}\\ \sigma \in \{+, -\}}} h_{K_j} \| (\nab w_j)_\sigma \|_{\grave{\Gamma}_{i {K_j}}}^2 \\
& \lesssim \sum_{\substack{i \in \{1, 2\}\\ K \in \mathcal{T}_{j, \Om{{i,n}}}}} \bigg(\| \nab w_j \|_{K}^2 + h_{K}^2\| D_x^2 w_j \|_{K}^2 \bigg) \lesssim h^{2p} \| D_x^{p+1} v \|_{\Om{0}}^2
\end{split} 
\label{interphest_energy_special_inveq}
\end{equation}
where we have used \eqref{scatraineqGamK_warmup} and standard local interpolation error estimates for Scott--Zhang interpolation operators. Using Lemma~\ref{lem:interphest_energy} and \eqref{interphest_energy_special_inveq} in \eqref{interphest_energy_special_init} shows \eqref{lemres:interphest_energy_special}.
\end{proof}
\end{lemma}

\noindent For $q \in \mathbb{N}$ and $n = 1, \dots, N$, we define the temporal interpolation operator $\tilde{I}^n : C(I_n) \to \mathcal{P}^{q}(I_n)$ by
\begin{subequations}
\begin{equation}
(\tilde{I}^n v)_n^- = v_n^- \label{ipdefnm}
\end{equation}  
and with the additional condition for $q \geq 1$,   
\begin{equation}
\int_{I_n} \tilde{I}^nv w \ud t = \int_{I_n} v w \ud t \quad \forall w \in \mathcal{P}^{q-1}(I_n) \label{ipdeforto}
\end{equation} 
\label{ipdef}
\end{subequations}
The operator $\tilde{I}^n$ is used in the proof of Theorem~\ref{aprithm} where an estimate of its interpolation error is used. We present this estimate in the following lemma:

\begin{lemma}[An interpolation error estimate in $\| \cdot \|_{\Om{0}, I_n}$] \label{lemInv}
Let $\tilde{I}^n$ be defined by \eqref{ipdef}. Then, for $q = 0, 1$, for any function $v : \Om{0} \times I_n \to \mathbb{R}$ with sufficient spatial and temporal regularity we have that $\tilde{I}^n$ is bounded and that
\begin{equation}
\|v - \tilde{I}^nv \|_{\Om{0}, I_n} \lesssim k_n^{q+1}\| \dot{v}^{(q+1)} \|_{\Om{0}, I_n}
\label{estvInv}
\end{equation}
where $\| v \|_{\Om{0}, I_n} = \max_{t \in I_n} \|v \|_{\Om{0}}$, $k_n = t_n - t_{n-1}$, and $\dot{v}^{(q+1)} = \partial^{q+1}v/\partial t^{q+1}$.

\begin{proof}
The proof is exactly as in~\cite{Eriksson1995} which we refer to for details. %First explicit expressions for $\tilde{I}^nv$ involving $v$ are derived from which boundedness of $\tilde{I}^n$ will follow. These expressions are then used to derive estimates for $v - \tilde{I}^nv$ from which (\ref{estvInv}) is derived. \\
\end{proof}
\end{lemma}

%\cleardoublepage

%\phantomsection
%\addcontentsline{toc}{section}{References}
%\nocite{*}
\bibliographystyle{IEEEtran}
\bibliography{bibliography}

\end{document}

%% file: figures/problemdomains.pdf_tex
%% Creator: Inkscape 1.0.2 (e86c8708, 2021-01-15), www.inkscape.org
%% Accompanies image file 'problemdomains.pdf' (pdf, eps, ps)
\begingroup%
  \makeatletter%
  \providecommand\color[2][]{%
    \errmessage{(Inkscape) Color is used for the text in Inkscape, but the package 'color.sty' is not loaded}%
    \renewcommand\color[2][]{}%
  }%
  \providecommand\transparent[1]{%
    \errmessage{(Inkscape) Transparency is used (non-zero) for the text in Inkscape, but the package 'transparent.sty' is not loaded}%
    \renewcommand\transparent[1]{}%
  }%
  \providecommand\rotatebox[2]{#2}%
  \newcommand*\fsize{\dimexpr\f@size pt\relax}%
  \newcommand*\lineheight[1]{\fontsize{\fsize}{#1\fsize}\selectfont}%
  \ifx\svgwidth\undefined%
    \setlength{\unitlength}{1348.68937852bp}%
    \ifx\svgscale\undefined%
      \relax%
    \else%
      \setlength{\unitlength}{\unitlength * \real{\svgscale}}%
    \fi%
  \else%
    \setlength{\unitlength}{\svgwidth}%
  \fi%
  \global\let\svgwidth\undefined%
  \global\let\svgscale\undefined%
  \makeatother%
  \begin{picture}(1,0.62242716)%
    \lineheight{1}%
    \setlength\tabcolsep{0pt}%
    \put(0.78294052,0.16163747){\color[rgb]{0,0,0}\makebox(0,0)[lt]{\lineheight{0}\smash{\begin{tabular}[t]{l}\large : $\Omega_1$\end{tabular}}}}%
    \put(0,0){\includegraphics[width=\unitlength,page=1]{problemdomains.pdf}}%
    \put(0.78294052,0.08378411){\color[rgb]{0,0,0}\makebox(0,0)[lt]{\lineheight{0}\smash{\begin{tabular}[t]{l}\large : $\Omega_2$\end{tabular}}}}%
    \put(0.78294052,0.00593075){\color[rgb]{0,0,0}\makebox(0,0)[lt]{\lineheight{0}\smash{\begin{tabular}[t]{l}\large : $\Gamma$\end{tabular}}}}%
    \put(0,0){\includegraphics[width=\unitlength,page=2]{problemdomains.pdf}}%
    \put(0.5003338,0.07073026){\color[rgb]{0,0,0}\makebox(0,0)[lt]{\lineheight{0}\smash{\begin{tabular}[t]{l}\large $\Omega_0$\end{tabular}}}}%
    \put(0.37285761,0.0923501){\color[rgb]{0,0,0}\makebox(0,0)[lt]{\lineheight{0}\smash{\begin{tabular}[t]{l}\large $\mu$\end{tabular}}}}%
    \put(0.31234314,0.29020725){\color[rgb]{0,0,0}\makebox(0,0)[lt]{\lineheight{0}\smash{\begin{tabular}[t]{l}\large $G$\end{tabular}}}}%
    \put(0,0){\includegraphics[width=\unitlength,page=3]{problemdomains.pdf}}%
  \end{picture}%
\endgroup%

%% file: figures/spacetimeslabs_dG0mesh.pdf_tex
%% Creator: Inkscape 1.0.2 (e86c8708, 2021-01-15), www.inkscape.org
%% Accompanies image file 'spacetimeslabs_dG0mesh.pdf' (pdf, eps, ps)
\begingroup%
  \makeatletter%
  \providecommand\color[2][]{%
    \errmessage{(Inkscape) Color is used for the text in Inkscape, but the package 'color.sty' is not loaded}%
    \renewcommand\color[2][]{}%
  }%
  \providecommand\transparent[1]{%
    \errmessage{(Inkscape) Transparency is used (non-zero) for the text in Inkscape, but the package 'transparent.sty' is not loaded}%
    \renewcommand\transparent[1]{}%
  }%
  \providecommand\rotatebox[2]{#2}%
  \newcommand*\fsize{\dimexpr\f@size pt\relax}%
  \newcommand*\lineheight[1]{\fontsize{\fsize}{#1\fsize}\selectfont}%
  \ifx\svgwidth\undefined%
    \setlength{\unitlength}{2069.29439705bp}%
    \ifx\svgscale\undefined%
      \relax%
    \else%
      \setlength{\unitlength}{\unitlength * \real{\svgscale}}%
    \fi%
  \else%
    \setlength{\unitlength}{\svgwidth}%
  \fi%
  \global\let\svgwidth\undefined%
  \global\let\svgscale\undefined%
  \makeatother%
  \begin{picture}(1,0.58581438)%
    \lineheight{1}%
    \setlength\tabcolsep{0pt}%
    \put(0.80868287,0.17397535){\color[rgb]{0,0,0}\makebox(0,0)[lt]{\lineheight{0}\smash{\begin{tabular}[t]{l}$S_{1,n-1}$\end{tabular}}}}%
    \put(0,0){\includegraphics[width=\unitlength,page=1]{spacetimeslabs_dG0mesh.pdf}}%
    \put(0.16278541,0.08884568){\color[rgb]{0,0,0}\makebox(0,0)[lt]{\lineheight{0}\smash{\begin{tabular}[t]{l}$x_2$\end{tabular}}}}%
    \put(0,0){\includegraphics[width=\unitlength,page=2]{spacetimeslabs_dG0mesh.pdf}}%
    \put(0.18791888,0.00217465){\color[rgb]{0,0,0}\makebox(0,0)[lt]{\lineheight{0}\smash{\begin{tabular}[t]{l}$x_1$\end{tabular}}}}%
    \put(-0.00076548,0.19035702){\color[rgb]{0,0,0}\makebox(0,0)[lt]{\lineheight{0}\smash{\begin{tabular}[t]{l}$t$\end{tabular}}}}%
    \put(0,0){\includegraphics[width=\unitlength,page=3]{spacetimeslabs_dG0mesh.pdf}}%
    \put(0.53396269,0.17687781){\color[rgb]{0,0,0}\makebox(0,0)[lt]{\lineheight{0}\smash{\begin{tabular}[t]{l}$S_{2,n-1}$\end{tabular}}}}%
    \put(0.14772529,0.03866927){\color[rgb]{0,0,0}\makebox(0,0)[lt]{\lineheight{0}\smash{\begin{tabular}[t]{l}$t_{n-2}$\end{tabular}}}}%
    \put(0,0){\includegraphics[width=\unitlength,page=4]{spacetimeslabs_dG0mesh.pdf}}%
    \put(0.40999625,0.40593848){\color[rgb]{0,0,0}\makebox(0,0)[lt]{\lineheight{0}\smash{\begin{tabular}[t]{l}$S_{1,n}$\end{tabular}}}}%
    \put(0,0){\includegraphics[width=\unitlength,page=5]{spacetimeslabs_dG0mesh.pdf}}%
    \put(0.17607703,0.50251425){\color[rgb]{0,0,0}\makebox(0,0)[lt]{\lineheight{0}\smash{\begin{tabular}[t]{l}$t_n$\end{tabular}}}}%
    \put(0,0){\includegraphics[width=\unitlength,page=6]{spacetimeslabs_dG0mesh.pdf}}%
    \put(0.74417928,0.42333864){\color[rgb]{0,0,0}\makebox(0,0)[lt]{\lineheight{0}\smash{\begin{tabular}[t]{l}$S_{2,n}$\end{tabular}}}}%
    \put(0.14772529,0.27063239){\color[rgb]{0,0,0}\makebox(0,0)[lt]{\lineheight{0}\smash{\begin{tabular}[t]{l}$t_{n-1}$\end{tabular}}}}%
    \put(0,0){\includegraphics[width=\unitlength,page=7]{spacetimeslabs_dG0mesh.pdf}}%
  \end{picture}%
\endgroup%

%% file: figures/fefuncspace.pdf_tex
%% Creator: Inkscape 0.48.5, www.inkscape.org
%% Accompanies image file 'fefuncspace.pdf' (pdf, eps, ps)
\begingroup%
  \makeatletter%
  \providecommand\color[2][]{%
    \errmessage{(Inkscape) Color is used for the text in Inkscape, but the package 'color.sty' is not loaded}%
    \renewcommand\color[2][]{}%
  }%
  \providecommand\transparent[1]{%
    \errmessage{(Inkscape) Transparency is used (non-zero) for the text in Inkscape, but the package 'transparent.sty' is not loaded}%
    \renewcommand\transparent[1]{}%
  }%
  \providecommand\rotatebox[2]{#2}%
  \ifx\svgwidth\undefined%
    \setlength{\unitlength}{861.750576bp}%
    \ifx\svgscale\undefined%
      \relax%
    \else%
      \setlength{\unitlength}{\unitlength * \real{\svgscale}}%
    \fi%
  \else%
    \setlength{\unitlength}{\svgwidth}%
  \fi%
  \global\let\svgwidth\undefined%
  \global\let\svgscale\undefined%
  \makeatother%
  \begin{picture}(1,0.49427199)%
    \put(0,0){\includegraphics[width=\unitlength]{fefuncspace.pdf}}%
    \put(0.2940156,0.00254569){\color[rgb]{0,0,0}\makebox(0,0)[lb]{\smash{$x$}}}%
    \put(0.06240954,0.46932096){\color[rgb]{0,0,0}\makebox(0,0)[lb]{\smash{$v(x,t)$}}}%
    \put(-0.0018268,0.06169315){\color[rgb]{0,0,0}\makebox(0,0)[lt]{\begin{minipage}{0.13910728\unitlength}\raggedright $0$\end{minipage}}}%
  \end{picture}%
\endgroup%